\documentclass[10pt]{amsart}
\usepackage{graphicx}
\usepackage{xcolor}
\usepackage{mathrsfs,bbm}  
\usepackage{amssymb, amsthm, thmtools}

\usepackage[pdftex, colorlinks=false, backref]{hyperref}

\usepackage{tikz} %To make figure.
\usepackage{pgflibraryarrows} %Extra arrow heads
\usetikzlibrary{patterns}

\numberwithin{equation}{section}

\declaretheorem[numberwithin = section]{theorem}
\declaretheorem[sibling = theorem]{lemma}
\declaretheorem[sibling = theorem]{corollary}
\declaretheorem[style = definition, sibling = theorem]{definition}
\declaretheorem[style = remark, sibling = theorem]{remark}

\newcommand{\Bm}{\mathscr{B}^-}
\newcommand{\Bp}{\mathscr{B}^+}
\newcommand{\Bpm}{\mathscr{B}^\pm}

\newcommand{\1}[1]{{\ensuremath{\mathbbm 1_{#1}}}}

\newcommand{\abs}[1]{{\left| #1 \right|}}

\newcommand{\seq}[1]{\left\{#1\right\}}

\DeclareMathOperator{\orm}{\mathrm{Orm}}

\newcommand{\norm}[1]{\left\|#1\right\|}

\newcommand{\R}{\mathbb{R}}

\newcommand{\F}{\widehat{\F}}%filtration

\newcommand{\ue}{u^\varepsilon}
\newcommand{\ve}{v^\varepsilon}

\newcommand{\test}{\varphi}

\newcommand{\marginlabel}[1]%
       {\mbox{}\marginpar{\raggedleft\hspace{0pt}\tiny{\textcolor{red}{#1}}}}

\begin{document}

\title{Path-dependent convex conservation laws}

\author[H. Hoel]{H. Hoel}
\address[H{\aa}kon Hoel]
{\newline Mathematics Institute of Computational Science and Engineering
\newline \'Ecole polytechnique f\'ed\'erale de Lausanne
\newline  EPFL / SB / MATH-CSQI, MA C1 644 (B{\^a}timent MA), Station 8
\newline CH-1015, Lausanne
\newline Switzerland}

\author[K. H. Karlsen]{K. H. Karlsen}
\address[Kenneth Hvistendahl Karlsen]
{\newline Department of mathematics
\newline University of Oslo
\newline P.O. Box 1053,  Blindern
\newline N--0316 Oslo, Norway} 
\email[]{k.h.karlsen@me.com }

\author[N. H. Risebro]{N. H. Risebro}
\address[Nils Henrik Risebro]
{\newline Department of mathematics
\newline University of Oslo
\newline P.O. Box 1053,  Blindern
\newline N--0316 Oslo, Norway} 
\email[]{nilshr@math.uio.no}

\author[E. B. Storr\o{}sten]{E. B. Storr\o{}sten}
\address[Erlend Briseid Storr\o{}sten]
{\newline Department of mathematics
\newline University of Oslo
\newline P.O. Box 1053, Blindern
\newline N--0316 Oslo, Norway} 
\email[]{erlenbs@math.uio.no}

% General info
% 35L65  	Conservation laws
% 60H15  	Stochastic partial differential equations


\thanks{This work was supported by the Research Council of Norway (project 250674/F20).}

\begin{abstract}
For scalar conservation laws driven by a rough path $z(t)$, in the sense of 
Lions, Perthame and Souganidis in \cite{Lions:2013aa}, we 
show that it is possible to replace $z(t)$ by a piecewise 
linear path, and still obtain the same solution at a given time, under 
the assumption of a convex flux function in one spatial dimension. 
This result is connected to the spatial regularity of solutions. 
We show that solutions are spatially Lipschitz continuous for a given set 
of times, depending on the path and the initial data. Fine properties of the 
map $z \mapsto u(\tau)$, for a fixed time $\tau$, are studied. 
We provide a detailed description of the properties of the 
rough path $z(t)$ that influences the solution. 
This description is extracted by a ``factorization" of 
the solution operator (at time $\tau$). In a companion 
paper \cite{Hoel-num:2017}, we make use of the observations herein to 
construct computationally efficient numerical methods.
\end{abstract}

\date{\today}

\maketitle

\section{Introduction}\label{sec:Intro}

We are interested in scalar conservation laws of the form 
\begin{equation}\label{eq:sscl}
  \begin{cases}
    \partial_tu + \partial_xf(u)\dot{z} = 0 &\mbox{ on } (0,T) \times \R,\\
    \hphantom{\partial_tu + x}
    u(0,x)=u_0(x) &\mbox{ for } x \in \R,
  \end{cases}
\end{equation}
where $0 <T<\infty$ is some fixed final time. The (rough) path 
$z:(0,\infty) \rightarrow \R$, the initial value $u_0:\R \rightarrow \R$, and 
the flux $f:\R \rightarrow \R$ are given functions, whereas $u$ is 
the unknown function that is sought. The time derivative of $z(t)$ is 
denoted by $\dot{z}$. Regarding the flux, the standing assumption is
\begin{equation}\label{eq:Af}\tag{\mbox{$\mathcal{A}_f$}}
	\mbox{$f\in C^2$ and $f$ is strictly convex.} 
\end{equation}

Bear in mind that \eqref{eq:sscl} reduces to a standard 
conservation law in the event $z(t) = t$. 
It is well known that such equations are well posed within the framework of 
Kru{\v{z}}kov entropy solutions \cite{Dafermos:2010fk} or, equivalently, kinetic 
solutions \cite{Perthame:2002qy}. More precisely, assuming 
for example $u_0 \in (L^\infty \cap L^1)(\R)$, there exists a unique 
function $u \in C([0,T];L^1(\R))$ satisfying $u(0,x) = u_0(x)$ and
\begin{equation}\label{eq:KruzEnt}
	\partial_tS(u) + \partial_xQ(u)\dot{z} \le 0
	\quad \text{in the sense of distributions},
\end{equation}
for all convex entropy, entropy-flux pairs $(S,Q)$, i.e., $S\in C^2$ 
convex and $Q' = S'f'$. 

If $z(t)$ is a Brownian path (i.e., a realization of a Brownian motion), 
then $z(t)$ is merely H\"older continuous (infinite variation) and the 
conservation law \eqref{eq:sscl} is no longer well defined; in this case one 
could replace \eqref{eq:sscl} by the stochastic 
partial differential equations (SPDE)
\begin{equation}\label{eq:stratonovich}
	\partial_tu + \partial_xf(u)\circ dz= 0,
\end{equation}
where $\circ$ denotes the Stratonovich differential. Aiming for a different 
approach, Lions, Perthame, and Souganidis \cite{Lions:2013aa} recently 
introduced a pathwise notion of entropy/kinetic solution to \eqref{eq:sscl}, defined 
for any $z \in C([0,T])$, which is consistent with the notion 
of Kru{\v{z}}kov entropy solution for regular paths $z(t)$.

According to \cite{Lions:2013aa}, 
$u \in L^\infty\left([0,T];L^1(\R)\right)\cap L^\infty(\R\times [0,T])$ 
is called a pathwise entropy/kinetic solution to \eqref{eq:sscl} provided 
there is a non-negative, bounded measure $m$ on $\R \times \R \times [0,T]$ 
such that for all $\rho \in C^\infty_c(\R)$, 
$\chi(\xi,u) := \1{0 < \xi < u}-\1{u < \xi < 0}$ satisfies
 \begin{equation}\label{eq:PathwiseEntSol}
 	\partial_t\int_\R \chi(\xi,u(t,x))\rho(x-f'(\xi)z(t))\,dx 
	= \int_\R \rho(x-f'(\xi)z(t))\partial_\xi m(t,x,\xi)\,dx,
 \end{equation}
in the weak sense on $\R \times [0,T]$. Informally, the motivation behind 
the notion of pathwise solutions \eqref{eq:PathwiseEntSol} comes 
from writing the usual kinetic formulation of \eqref{eq:sscl},
$$
d \chi + f'(\xi) \cdot \nabla \chi\, \dot{z}= \partial_\xi m,
$$
for a bounded measure $m(t,x,\xi)\ge 0$ and a function $u(t,x)$ (entropy solution) such 
that $\chi=\chi(t,x,\xi):=\chi(\xi, u(t,x))$. Next, one uses 
the ``method of characteristics" to remove the rough path. 
The result is that the function
$$
v=v(t,x,\xi):=\chi\left(t,x+f'(\xi)z(t),\xi\right)
$$
satisfies the following kinetic equation 
without the rough drift term:
\begin{equation}\label{eq:transformed}
	\begin{split}
		\partial_t v  & = \left(\partial_\xi  m\right) \left(t,x+f'(\xi)z(t),\xi \right)
		\\ &=\partial_\xi \Bigl(m(t,x+f'(\xi) z(t),\xi)\Bigr) 
		\\ & \qquad - f''(\xi)z(t)\, 
		\left(\partial_x m\right)\left(t,x+f'(\xi) z(t),\xi\right).
	\end{split}
\end{equation}
The defining equation \eqref{eq:PathwiseEntSol} constitutes 
a weak formulation of \eqref{eq:transformed}. Various results concerning 
existence, uniqueness, and stability of pathwise solutions are found in the works 
\cite{Lions:2012fk,Lions:2013aa,Lions:2014aa}. The theory of pathwise solutions has been further 
developed in \cite{Gess:2014aa,Gess:2015aa}, see also \cite{Hofmanova:2016aa} 
and \cite{Bailleul:2015aa,Deya:2016aa}.

It is proved in \cite{Lions:2013aa} that the pathwise 
solution is stable with respect to uniform convergence 
of the path. More precisely, assuming $u_0 \in BV(\R)$ ($u_0$ is of 
bounded variation), we have the following 
result \cite[Theorem~3.2]{Lions:2013aa}: Let $u^i$ be 
the pathwise entropy/kinetic solution of \eqref{eq:sscl} with path $z^i$ 
and initial condition $u^i_0$, for $i = 1,2$. 
Then there is a constant $C > 0$ such that for $t\in [0,T]$,
 \begin{equation}\label{eq:stabilityPes}
    \begin{split}
      &\norm{u^1(t) - u^2(t)}_{L^1(\R)} \le \norm{u^1_0 - u^2_0}_{L^1(\R)} \\
      &\hphantom{XXX}+C\Bigg[\,  \norm{f'}_\infty\left(\norm{u^1_0}_{BV(\R)} 
      + \norm{u^2_0}_{BV(\R)}\right)\abs{z^1(t)-z^2(t)}\\
      &\hphantom{XXXXXX} +\sqrt{\sup_{s\in(0,t)}\abs{z^1(s) 
      -z^2(s)}\norm{f''}_\infty\left(\norm{u^1_0}_{L^2(\R)}^2
      +\norm{u^2_0}_{L^2(\R)}^2\right)}\, \Bigg],
    \end{split}
  \end{equation}
Consequently, given a sequence of regular (say, Lipschitz) 
paths $\seq{z^n}_{n \geq 0}$ converging uniformly to $z$ 
as $n \rightarrow \infty$, the corresponding 
Kru{\v{z}}kov entropy solutions $\seq{u^n}_{n \geq 0}$ of 
\eqref{eq:sscl} converges to the entropy/kinetic pathwise solution 
$u$ in $C([0,T];L^1(\R))$ as $n \rightarrow \infty$. As such, the 
interpretation of \eqref{eq:sscl} in terms of \eqref{eq:PathwiseEntSol} is 
associated with the Stratonovich interpretation of \eqref{eq:sscl}. 
In view of the consistency between \eqref{eq:KruzEnt} 
and \eqref{eq:PathwiseEntSol}, in what follows we will
refer to Kru{\v{z}}kov entropy solutions and pathwise 
entropy/kinetic solutions both simply as \emph{entropy solutions}. 

In this paper, we approximate the entropy solution $u$ of \eqref{eq:sscl} 
by a sequence $\seq{u^n}_{n \geq 0}$ of solutions utilizing 
piecewise linear approximations $\seq{z^n}_{n \geq 0}$ 
of the rough path $z \in C([0,T])$. The ``continuous dependence on the data" 
estimate \eqref{eq:stabilityPes} ensures that this approximation converges 
to the correct solution of \eqref{eq:sscl}. A motivation for exploring 
such approximations is their relevance to numerical methods.  
The computational difficulties associated with 
solving \eqref{eq:sscl} numerically stem from the infinite variation 
of the rough path $z(t)$. This forces the time step to be very 
small due to the well-known CFL stability condition, 
linking the temporal and spatial discretization parameters. 
Our main result, valid for convex flux functions $f$, states that it is 
possible to replace the rough path by 
a piecewise linear path of finite variation, and still obtain the 
same solution at a fixed time. In \cite{Hoel-num:2017} 
we make use of this result to construct computationally 
efficient finite volume methods.

Let us discuss in more detail our results relating to 
path-dependence of entropy solutions. Suppose the path $z(t)$ is piecewise 
linear and continuous, and let $u$ be the corresponding entropy solution 
to \eqref{eq:sscl}. Fix a time $\tau \in [0,T]$. We seek the ``simplest" 
path $\tilde{z}$ such that the corresponding 
solution $\tilde{u}$ satisfies $\tilde{u}(\tau,\cdot) = u(\tau,\cdot)$.  
To motivate, fix a time interval $[t_1,t_2]$ with $0 \le t_1 < t_2 \le \tau$, and 
suppose for simplicity that $z(t) \geq z(t_1)$ on $[t_1,t_2]$. 
Let $v$ be the entropy solution to
\begin{displaymath}
 \partial_tv + \partial_xf(v) = 0, \qquad v(z(t_1),x) = u(t_1,x).
\end{displaymath}
Then $u(t,x) = v(z(t),x)$ on $[t_1,t_2]$ in either of the two following cases:
\begin{itemize}
 \item[(i)] $z(t)$ is monotone on $[t_1,t_2]$.
 \item[(ii)] $v$ is a classical solution (no shocks) on $\seq{z(t)\,:t \in [t_1,t_2]}$.
\end{itemize}
Consequently, we may replace parts of the path $z$ satisfying either (i) or (ii) by 
straight line segments, i.e., $z$ by the path  
\begin{displaymath}
  \tilde{z}(t) := \begin{cases}
                  z(t_1) + (t-t_1)\frac{z(t_2)-z(t_1)}{t_2-t_1} \mbox{ for $t \in [t_1,t_2]$,}\\
		  z(t) \mbox{ else.}
                 \end{cases}
\end{displaymath} 
As the solution $\tilde u$ (corresponding to $\tilde z$) satisfies 
$\tilde{u}(t_2) =  v(\tilde z(t_2)) = v(z(t_2)) = u(t_2)$ it follows that $u(\tau) = \tilde u(\tau)$. 
In view of this, we can ``simplify" the path $z(t)$ by replacing the parts where 
it satisfies either $(i)$ or $(ii)$ by straight line segments. It is easy to determine 
which parts satisfies (i), i.e., where the path is monotone. 
Considering (ii), we need to determine the parts of the path $z$ 
on which $u$ (and $v$) is a classical solution (without shocks). To this end, let us recall 
the so-called Ole{\u\i}nik estimate for strictly convex fluxes 
$f$ (and $\dot{z}=1$) \cite{Dafermos:2010fk}:
$$
\frac{f'(u(t,y))-f'(u(t,x))}{y-x} \le \frac{1}{t}, \qquad y>x, \  t>0.
$$
For example with $f=u^2/2$, the only admissible shocks are 
those for which the left value is larger than the right.
Similarly, with $f=u^2/2$ and $\dot{z}=-1$, only upward jumping 
shocks are admissible. When $z'(t)$ oscillates (i.e., takes 
on positive and negative values, say $\pm 1$), one observes 
that shocks in $u$ can only exist when the path $z$ takes 
on values not already assumed at some earlier point in time. 
Consequently, starting from a piecewise linear path $z$, it is 
possible to ``inductively" construct a new path $\tilde z(t)$, with smaller
total variation, by replacing appropriate parts of $z(t)$ by linear segments. 
This inductive procedure, which we describe in detail in an
upcoming section, gives rise to a ``minimal'' path associated with 
the final time $\tau$ and the initial data $u_0$. We refer to 
the resulting path as the \emph{oscillating running min/max} 
path and label it $\orm_{\tau,u_0}(z)$ (See Definition~\ref{def:Orm} below). 
Setting $\tilde{z} = \orm_{\tau,u_0}(z)$, we have 
$\tilde{u}(\tau) = u(\tau)$. Note that the application of the 
Ole{\u\i}nik estimate depended on $\dot{z}$ being piecewise constant. 
However, for sufficiently smooth $u_0$, it turns out that that 
$z \mapsto \orm_{\tau,u_0}(z)$ is well-defined for any continuous path $z$.  
Indeed, for a general path $z$, we proceed by suitable 
piecewise linear approximation. 

This approach may be viewed as a factorization of the solution operator. 
To see how this is related to the construction of the map 
$z \mapsto \orm_{\tau,u_0}(z)$, we identify two paths $z$ and $\tilde z$ as long as 
$u(\tau) = \tilde u(\tau)$, where $z \mapsto u(\tau)$ and $\tilde z \mapsto \tilde u(\tau)$. 
This naturally leads to a factorization of the solution operator 
(one for each fixed time $\tau$) as a composition 
of a quotient map and an injective map, see~Figure~\ref{fig:Factorization}. 
Up to precomposition by a nondecreasing function, the 
quotient map may be identified with the map $z \mapsto \orm_{\tau,u_0}(z)$. 
The injective map is associated with the solution operator 
restricted to piecewise linear paths. 

Another question raised in this work is related to the optimal 
choice of paths relative to the continuous dependence estimate \eqref{eq:stabilityPes}. 
To be more precise, in view of the above discussion, there exists for 
each path $z$ a multitude of paths $\tilde z$ such that, for a 
given time $\tau  \in [0,T]$ and initial 
condition $u_0$, the corresponding solutions $u$ and $\tilde u$ 
to \eqref{eq:sscl} satisfy $u(\tau) = \tilde u(\tau)$. In order to 
improve \eqref{eq:stabilityPes} one may search for 
paths $\tilde z_1, \tilde z_2$ satisfying 
$u^1(\tau) = \tilde u^1(\tau), u^2(\tau) = \tilde u^2(\tau)$ such that 
\begin{displaymath}
	 \sup_{0 \leq t \leq \tau} \seq{\abs{\tilde z_1(t)-\tilde z_2(t)}} 
	 \mbox{ is as small as possible.}
\end{displaymath}
In Theorem~\ref{thm:EqDist} below, it is shown that this 
minimization problem may be bounded in terms of a 
second minimization problem solvable by dynamic programming.

Before ending this introduction, we mention that recently many researchers 
studied the effect of adding randomness to conservation laws and 
other related nonlinear partial differential equations. This includes 
stochastic transport equations, which bears some 
resemblance to \eqref{eq:stratonovich},
\begin{equation}\label{eq:transport}
	du +  b(x)\cdot \nabla u  \,dt + \nabla u \circ dW(t) =0,
\end{equation}
where $b(x)$ is a low-regularity velocity field 
and the ``transportation noise" is driven by a 
Wiener process $W(t)$. For some representative results, see 
e.g.~\cite{Attanasio:2011fj,Flandoli:2010yq,Mohammed:2015aa,Neves:2015aa}.
In a different direction, many mathematical papers 
\cite{Bauzet:2012kx,Biswas:2014gd,Chen:2011fk,Debussche:2010fk,Debussche:2015aa,Debussche:2016aa,E:2000lq,Hofmanova:2013aa,Kim2003,Karlsen:2015ab,Feng:2008ul,Vallet:2009uq,Vallet:2000ys} have studied the effect of It\^{o} stochastic forcing on conservation laws,
\begin{equation}\label{eq:source}
	du +  \nabla\cdot f(u)  \,dt=\sigma(u) \, dW(t),
\end{equation}
where $f,\sigma$ are nonlinear functions and $W(t)$ is a (finite or infinite dimensional) 
Wiener process. Numerical methods are looked at in 
\cite{Audusse:2015aa,Bauzet:2016ab,Bauzet:2016aa,Bauzet:2015aa,Holden:1997fk,Karlsen:2016aa,Dotti:2016aa,Dotti:2016ab,Kroker:2012fk}.

The remaining part of this paper is organized as follows: 
In Section~\ref{sec:mainResults}, the main results of the paper are presented, without proofs, along with the notation necessary to make the statements precise. In Section~\ref{seq:Proofs} proofs of the given results are presented.

\section{Main results}\label{sec:mainResults}
To state the main results precisely, we introduce some notation and definitions. 
The regularity of $u_0$ is quantified by two numbers 
$0 \le M_+,M_- \le \infty$ satisfying
\begin{equation}\label{eq:LowUpOnInitial}
 -M_- \le \frac{f'(u_0(y))-f'(u_0(x))}{y-x} \le M_+,
 \qquad x,y\in \R, \, x<y.
\end{equation}
Denote by $C_0([0,T]) = \seq{z \in C([0,T])\,:\,z(0) = 0}$ the space of 
continuous paths starting at the origin. For a given path 
$z \in C_0([0,T])$, we introduce the ``truncated'' running 
min/max functions, which are defined by
\begin{equation}\label{eq:TruncRunningMaxMin}
  \rho^+_z(t) := \max \seq{\frac{1}{M_-},\max_{0 \le s \le t}\seq{z(s)}} 
  \mbox{ and } \rho^-_z(t) := \min \seq{-\frac{1}{M_+},\min_{0 \le s \le t}\seq{z(s)}},
\end{equation}
cf.~Figure~\ref{fig:def_plot}. Here we use the 
convention that $0^{-1} = \infty$. Consequently, if $M_- = 0$, 
then $\rho^+_z(t) = \infty$ for all $t  \in [0,\tau]$. A path $z \in C_0([0,T])$ is called
\emph{piecewise linear} if there is a finite sequence
$\seq{t_n}_{n=0}^N$ with $0=t_0<t_1<\cdots<t_N=T$, such that
\begin{equation*}
  z(t)=z(t_n) +
  \left(t-t_n\right)\frac{z(t_{n+1})-z(t_n)}{t_{n+1}-t_n}\qquad
  \text{for $t\in [t_n,t_{n+1}]$.}
\end{equation*}
Furthermore, we define the sets where 
$\rho^\pm_z$ is strictly increasing/decreasing by
\begin{equation}\label{eq:BpBmDef}
 \begin{split}
 \Bp_z := \seq{t \in [0,T]\,:\, \inf \seq{s \in [0,T]\,:\, \rho_z^+(s) \geq \rho_z^+(t)} = t},\\
 \Bm_z := \seq{t \in [0,T]\,:\, \inf \seq{s \in [0,T]\,:\, \rho_z^-(s) \le \rho_z^-(t)} = t}.
 \end{split}
\end{equation}

The next lemma summarizes the essential properties of these sets.

%Defining macro, since I want the same lemma to appear in the proof section
\newcommand{\TextBpBmProp}{
 Let $\Bpm_z$ be defined by \eqref{eq:BpBmDef}. Then
 \begin{itemize}
  \item[(i)] $\Bpm_z \subset \seq{t \in [0,T]\,:\, z(t) = \rho_z^\pm(t)} \cup \seq{0}$. 
  Furthermore $\Bp_z \cap \Bm_z = \seq{0}$.  
  \item[(ii)] $\Bpm_z$ are closed with respect to increasing sequences, i.e., 
  if $\seq{t_n}_{n \geq 0} \subset \Bpm_z$ satsifies $t_n \uparrow t$ for 
  some $t \geq 0$, then $t \in \Bpm_z$. Hence 
  \begin{equation*}
   \sup\seq{\Bpm_z \cap [0,\tau]} \in \Bpm_z \mbox{ for any } \tau \in [0,T].
  \end{equation*}
  \item[(iii)] Let $0 \le t_1 < t_2 \le T$. Suppose $(t_1,t_2] \cap \Bpm_z = \emptyset$. Then $\rho_z^\pm$ is constant on $[t_1,t_2]$ respectively.
  \item[(iv)] Suppose $\rho_z^\pm$ is left differentiable on $(0,T]$. Then 
  \begin{equation*}
   \begin{split}
   \Bp_z &= \mathrm{cl}_-\left(\seq{t \in (0,T]\,:\, 
   \partial_- \rho_z^+(t) > 0}\right) \cup \seq{0}, \\
   \Bm_z &= \mathrm{cl}_-\left(\seq{t \in (0,T]\,:\, 
   \partial_- \rho_z^-(t) < 0}\right) \cup \seq{0},
   \end{split}
  \end{equation*}
  where $\partial_-$ denotes the left derivative and $\mathrm{cl}_-$ 
  denotes the closure with respect to increasing sequences. 
  Consequently, for piecewise linear $z$ there exists 
  $0 \le N^\pm < \infty$ and $0 \le s_1^\pm 
  < t_1^\pm < \dots < s_{N^\pm}^\pm < t_{N^\pm}^\pm \le T$ such that 
  \begin{equation*}
   \Bpm_z \setminus \seq{0} = \bigcup_{n = 1}^{N^\pm} (s_n^\pm,t_n^\pm],
  \end{equation*}
  where we use the convention that the union is empty if $N^\pm = 0$.
 \end{itemize}
 }
\begin{lemma}\label{lemma:BpBmProp}
 \TextBpBmProp
\end{lemma} 
We may now give the precise definition of the Oscillating Running Min/Max.
\begin{definition}[Oscillating Running Min/Max]\label{def:Orm}
 Given a path $z \in C_0([0,T])$, define the 
 sequence $\seq{\tau_n}_{n \geq 0}$ inductively by 
 \begin{align*}
  \tau_0& = \tau, \\
   \tau_{n+1}& = \begin{cases}
                \max\seq{\Bp_z \cap [0,\tau_n]} &\mbox{ if } \tau_n \in \Bm_z,\\
 		\max\seq{\Bm_z \cap [0,\tau_n]}  &\mbox{ if } \tau_n \in \Bp_z,\\
		\max\seq{(\Bp_z \cup \Bm_z) \cap [0,\tau_n]} 
		& \mbox{ if } \tau_n \notin \Bp_z \cup \Bm_z,
                \end{cases}
 \end{align*}
 for $n = 0,1,2,\dots$. If there exists an integer $0 \le N < \infty$ such that $\tau_N = 0$, 
 then we define $\orm_{\tau,M_\pm}(z)$ (Oscillating Running Min/Max) as 
 the piecewise linear interpolation of $\seq{(\tau_n,z(\tau_n))}_{n = 0}^N$. 
\end{definition}

\begin{figure}[h]\label{fig:def_plot}
 \centering
  \includegraphics[width=0.9\textwidth]{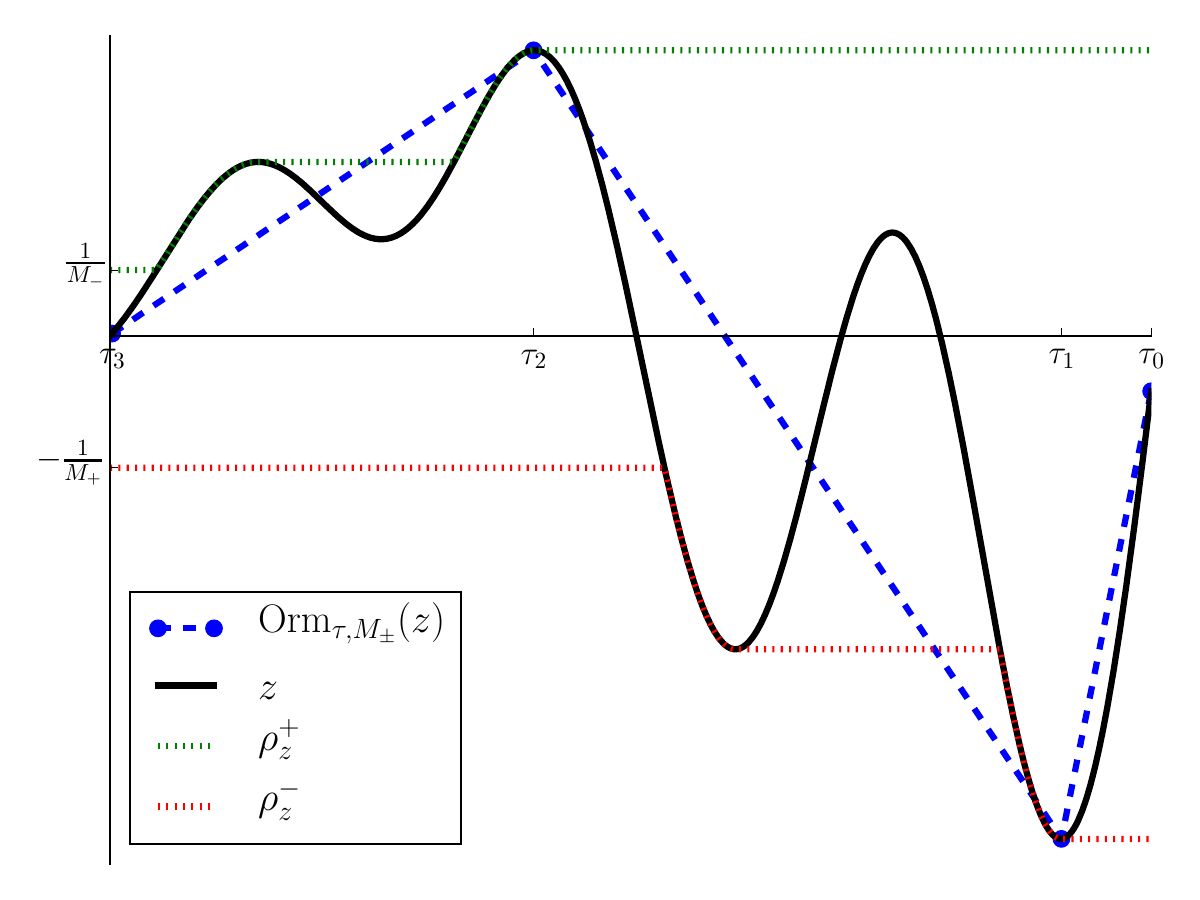}
 \caption{Illustration of $\rho^+_z,\rho^-_z$ and 
 $\orm_{\tau,M_\pm}(z)$ for a given path $z$.}
\end{figure}

Even though the Oscillating Running Min/Max only 
depends on $z,\tau,M_-,M_+$, we are often interested in 
a specific initial condition $u_0$ satisfying \eqref{eq:LowUpOnInitial} 
for some given numbers $0 \le M_-,M_+ \le \infty$. In such situations we often 
write $\orm_{\tau,u_0}(z)$ instead of $\orm_{\tau,M_\pm}(z)$. 
Let us mention that $\orm_{\tau,M_\pm}(z)$ is well defined for 
any path $z \in C_0([0,T])$, given that $0 \le \min\seq{M_+,M_-} < \infty$, 
see Lemma~\ref{lemma:FiniteOrm}. In view of the above discussion, there 
emerges a natural equivalence relation on the set of paths. 
For convenience, the relation is here defined on an arbitrary interval. 

\begin{definition}[Equivalence of paths]\label{def:EquivOfPaths} 
Fix a time interval $[t_1,t_2]\subset [0,\infty)$ and two paths 
$z_1,z_2 \in C([t_1,t_2])$. Let $u^i$ be the entropy solution to 
\begin{equation*}
 \partial_t u^i +  \partial_x f(u^i)\dot{z}_i = 0, 
 \qquad u^i(t_1) = u_0\in (L^1 \cap L^\infty)(\R),
\end{equation*} 
on $[t_1,t_2]$ for $i = 1,2$. If $u^1(t_2) = u^2(t_2)$, we say that $z_1$ is 
equivalent to $z_2$, written $z_1 \sim z_2$, 
on $[t_1,t_2]$ with initial condition $u_0$.
\end{definition}

We are now ready to state the result alluded to above.
\newcommand{\TextOrmEquiv}{Let $f$ satisfy \eqref{eq:Af}, $z \in C_0([0,\tau])$. 
If $\orm_{\tau,M_\pm}(z)$ is well-defined, then 
$z \sim \orm_{\tau,M_\pm}(z)$ on $[0,\tau]$ for any 
$u_0 \in (L^1 \cap L^\infty)(\R)$ satisfying \eqref{eq:LowUpOnInitial}.}
\begin{theorem}\label{thm:OrmEquiv}
\TextOrmEquiv
\end{theorem}

As mentioned above, for piecewise linear paths, the 
Ole{\u\i}nik estimate implies that the solution 
$u(t)$ is (spatially Lipschitz) continuous for $t$ in certain regions of the path. 
In the following theorem this result is extended, via Theorem~\ref{thm:OrmEquiv},
to the case of a general path $z \in C_0([0,T])$.

\newcommand{\TextOSB}{Assume $f$ satisfies \eqref{eq:Af}, 
$z \in C_0([0,T])$ and $u_0 \in (L^1 \cap L^\infty)(\R)$. 
Let $u$ be the entropy solution to \eqref{eq:sscl}. Then, for any $t \in [0,T]$,
 \begin{equation*}
     -\frac{1}{\rho^+_z(t)-z(t)} \le \frac{f'(u(t,y\pm))-f'(u(t,x\pm))}{y-x} 
     \le \frac{1}{z(t)-\rho^-_z(t)},
 \end{equation*}
 for all $-\infty < x < y < \infty$. Here, $u(t,x\pm)$ denotes respectively
 the right and left limits. We apply 
 the convention $(\infty)^{-1} = 0$ and $(0)^{-1} = \infty$.}

\begin{theorem}\label{thm:OSB}
 \TextOSB 
\end{theorem}

\begin{remark}
 Apriori, the left/right limits should be interpreted as essential 
 limits and the statement should be restricted 
 to points $-\infty < x < y < \infty$ such that these limits exists. 
 However, whenever the lower or upper bound is 
 finite, it implies that $u(t,\cdot)$ belongs to $BV_{\mathrm{loc}}(\R)$ 
 and the left/right limits exist in the classical sense.
\end{remark}

\begin{remark}
In \cite{Gassiat:2016aa}\footnote{We became aware 
of the work \cite{Gassiat:2016aa} in the final stage of writing this paper.}, the 
authors investigate regularity properties of solutions to the equation
$$
dv+\frac{1}{2}\abs{Dv}^2\circ dz = F(x,v,Dv,D^2v)\, dt,
$$
where $z$ is a continuous path, and $F$ is a 
nonlinear function meeting the standard assumptions 
from the theory of viscosity solutions of fully nonlinear 
degenerate parabolic PDEs. An $L^\infty$-bound
on the second derivative $D^2v$ is 
established in \cite{Gassiat:2016aa}. 
In the special case
\begin{equation}\label{eq:gess-1D}
	du+\partial_x\left(\frac{u^2}{2}\right)\circ dz = 0,
	\qquad u= \partial_x v,
\end{equation}
this estimate reduces to the Lipschitz ($W^{1,\infty}$) bound
$$
\underset {x\neq y}{\mathrm{ess}\, \sup}\abs{\frac{u(x)-u(y)}{x-y}}\le 
\frac{1}{\max\limits_{s\in [0,t]} z(s)- z(t)} \vee
\frac{1}{z(t)-\min\limits_{s\in [0,t]} z(s)}.
$$
This estimate is similar to the one provided by Theorem \ref{thm:OSB}, 
which in the special case $f(u)=u^2/2$ can be recast as
$$
-\frac{1}{\rho_z^+(t)-z(t)}\le \frac{u(t,x)-u(t,y)}{x-y}\le 
\frac{1}{z(t)-\rho_z^-(t)}, \qquad \text{for a.e.~$x,y\in \R$}.
$$

Although the results are similar, both relying on the strict 
convexity of the flux but with the one in 
\cite{Gassiat:2016aa} restricted to $f=u^2/2$, the 
proofs are different. We work at the level of conservation 
laws and use the method of generalized characteristics. 
The argument in \cite{Gassiat:2016aa} relies on 
semiconvexity preservation properties of 
Hamilton-Jacobi equations. 
\end{remark}

In view of Theorem~\ref{thm:OrmEquiv}, the equivalence class 
of a given path is nontrivial. The following result yields 
a condition sufficient for two paths to be equivalent. Let 
\begin{displaymath}
 \mathscr{A}_\tau := \seq{\alpha:[0,\tau] \rightarrow 
 [0,\tau]\,:\, \alpha \mbox{ is nondecreasing and surjective}}.
\end{displaymath}

\newcommand{\TextGeneralEquiv}{
 Let $f$ satisfy \eqref{eq:Af}. Fix $z_1, z_2 \in C_0([0,\tau])$ 
 such that $z_1(\tau) = z_2(\tau)$. 
 Suppose there exist paths $\alpha_i \in \mathscr{A}_\tau$, $i =1,2$, such that 
 \begin{equation}\label{eq:path-relation}
  \rho^+_{z_1} \circ \alpha_1 = \rho^+_{z_2} \circ \alpha_2 \qquad 
  \mbox{ and } \qquad \rho^-_{z_1} \circ \alpha_1 = \rho^-_{z_2} \circ \alpha_2.
 \end{equation} 
Then $z_1 \sim z_2$ on $[0,\tau]$ for all $u_0 \in 
(L^\infty \cap L^1)(\R)$ satisfying \eqref{eq:LowUpOnInitial}.}

\begin{theorem}\label{thm:GeneralEquiv}
 \TextGeneralEquiv
\end{theorem}

\begin{remark}
 We note that the existence of $\alpha_1,\alpha_2$ is closely 
 related to the problem of optimal transport (on $\R$) \cite{Villani:2003vn}. 
 Here we have two (continuous) transference plans, represented 
 by $\alpha_1, \alpha_2$, which should satisfy two transportation problems. 
 Recall that $\rho^+_{z_i}$ is constant whenever $\rho_{z_i}^-$ is 
 decreasing, while $\rho^-_{z_i}$ is constant as long as 
 $\rho_{z_i}^+$ is increasing, $i = 1,2$.
\end{remark}

 It seems likely that the condition \eqref{eq:path-relation} is 
 also necessary, at least on a more restricted space 
 of paths, cf.~Lemma~\ref{lemma:PiecewisePrecomEquiv}. 
 For $z_1,z_2$ as in Theorem~\ref{thm:GeneralEquiv}, we write 
 $z_1 \sim_\circ z_2$ if there exist $\alpha_1,\alpha_2 \in \mathscr{A}_\tau$, such 
 that $\rho^\pm_{z_1} \circ \alpha_1 = \rho^\pm_{z_2} \circ \alpha_2$. 
 It is obvious that the relation $\sim_\circ$ is both reflexive and 
 symmetric, i.e., that $z_1 \sim_\circ z_1$ and $z_1 \sim_\circ z_2$ 
 implies $z_2 \sim_\circ z_1$. To check that it is transitive, fix 
 $z_1,z_2,z_3 \in C_0([0,\tau])$ such that 
 $z_1 \sim_\circ z_2$ and $z_2 \sim_\circ z_3$, i.e., suppose 
 there exist nondecreasing surjective maps 
 $\alpha_i,\beta_i$, $i = 1,2$, such that 
 \begin{equation*}
  \rho^\pm_{z_1} \circ \alpha_1 = \rho^\pm_{z_2} \circ \alpha_2 
  \quad \mbox{ and } \quad \rho^\pm_{z_2} \circ \beta_1 = \rho^\pm_{z_3} \circ \beta_2.
 \end{equation*}
 Then there exist, at last in the piecewise linear setting, nondecreasing  
 surjective maps $\zeta_1,\zeta_2$ such that 
 $\alpha_2 \circ \zeta_1 = \beta_1 \circ \zeta_2$, 
 cf.~Lemma~\ref{lemma:PiecewisePrecomEquiv}. Hence,
 \begin{equation*}
  \rho^\pm_{z_1}\circ (\alpha_1\circ\zeta_1) = \rho^\pm_{z_2}\circ\alpha_2\circ\zeta_1 
		= \rho^\pm_{z_2}\circ\beta_1\circ\zeta_2 
		= \rho^\pm_{z_3}\circ(\beta_2\circ\zeta_2),
 \end{equation*}
and so $z_1 \sim_\circ z_3$.

For a path $z \in C_0([0,\tau])$, we denote its equivalence class with 
respect to $\tau$ and $M_\pm$ by $[z]_{\tau,M_\pm}$, that is, 
\begin{displaymath}
	[z]_{\tau,M_\pm} = \seq{\tilde{z} \in C_0([0,\tau])\,:\, \tilde{z} \sim z 
	\mbox{ on $[0,\tau]$ for any $u_0$ 
	satisfying \eqref{eq:LowUpOnInitial}}}.
\end{displaymath}
Fix $0 \leq M_+^i,M_-^i \leq \infty$, $i = 1,2$. 
Regarding the continuous dependence estimate \eqref{eq:stabilityPes}, one 
may exchange the uniform distance between two paths by the 
distance between two equivalence classes:
\begin{multline}\label{eq:eqDist}
 	\norm{[z_1]_{\tau,M^1_\pm}-[z_2]_{\tau,M^2_\pm}}_\infty  \\
  	:=\inf\seq{\norm{\zeta_1-\zeta_2}_\infty \,:\,\zeta_1,\zeta_2 \in C_0([0,\tau]), 
  	\zeta_1 \overset{\tau,M^1_\pm}{\sim} z_1 , \zeta_2 \overset{\tau,M^2_\pm}{\sim} z_2},
\end{multline}
where $\zeta \overset{\tau,M_\pm^i}{\sim} z$ is shorthand notation 
for $\zeta \sim z$ on $[0,\tau]$ for any initial condition 
$u_0$ satisfying \eqref{eq:LowUpOnInitial} with $M_\pm = M_\pm^j$. 
Therefore, \eqref{eq:stabilityPes} may be replaced by
\begin{equation}\label{eq:stabilityPesII}
    \begin{split}
      &\norm{u^1(\tau) - u^2(\tau)}_{L^1(\R)} \le \norm{u^1_0 - u^2_0}_{L^1(\R)} \\
      &\hphantom{XXX}+C\Bigg[\,  \norm{f'}_\infty\left(\norm{u^1_0}_{BV(\R)} 
      + \norm{u^2_0}_{BV(\R)}\right)\abs{z^1(\tau)-z^2(\tau)}\\
      &\hphantom{XXXXXX} +\sqrt{\norm{[z^1]_{\tau,M^1_\pm} 
      -[z^2]_{\tau,M^2_\pm}}_{\infty}\norm{f''}_\infty\left(\norm{u^1_0}_{L^2(\R)}^2
      +\norm{u^2_0}_{L^2(\R)}^2\right)}\, \Bigg],
    \end{split}
\end{equation}
for $u_0^i$ satisfying $\eqref{eq:LowUpOnInitial}$ with 
$M_\pm = M_\pm^i$, $i = 1,2$. Assuming 
$M_\pm = M_\pm^1 = M_\pm^2$, in view of 
Theorem~\ref{thm:GeneralEquiv}, one may 
hypothesize that the distance \eqref{eq:eqDist} can be 
estimated in terms of the minimization problem
\begin{multline*}
	\norm{[z_1]_{\tau,M_\pm}-[z_2]_{\tau,M_\pm}}_\infty \\
	\leq \inf_{\alpha_1,\alpha_2 \in \mathscr{A}_\tau}
	\seq{\max_\pm \seq{\norm{\rho_{z_1}^\pm \circ \alpha_1
	-\rho_{z_2}^\pm\circ \alpha_2}_\infty, \abs{z_1(\tau)-z_2(\tau)}}}.
\end{multline*}
 Our next result shows that this is indeed the case. 
 To make the statement precise, we need to introduce some notation. 
 Suppose $z \in C_0([0,\tau])$ is such that $\orm_{\tau,M_\pm}(z)$ 
 is well defined. Denote by $\mathcal{T}_z := \seq{\tau_n}_{n= 0}^N$ 
 the interpolation points associated with $\orm_{\tau,M_\pm}(z)$, 
 cf.~Definition~\ref{def:Orm}, and set 
 \begin{equation}\label{eq:TpmDef}
 	\mathcal{T}_{z}^\pm :=  \mathcal{T}_{z} \cap \Bpm_{z} \setminus \seq{0}.
 \end{equation}
 According to Lemma~\ref{lemma:BpBmProp}~(i),
\begin{displaymath}
  \mathcal{T}_{z}^\pm = \seq{\tau_n\,:\,0 \leq n \leq N-1 
  \mbox{ such that } z(\tau_n) = \rho_z^\pm(\tau_n)}.
 \end{displaymath}
For any $\alpha \in \mathscr{A}_\tau$, $\mathcal{T}^\pm_{z \circ \alpha} 
= \seq{\alpha^{-1}(t)\,:\, t \in \mathcal{T}_z^\pm}$, where $\alpha^{-1}$ 
denotes the generalized inverse of $\alpha$, i.e., 
\begin{equation}\label{eq:gen-inv}
	\alpha^{-1}(t) := \inf\seq{s\in [0,\tau]\,:\, \alpha(s) \geq t}.
\end{equation} 

 \begin{theorem}\label{thm:EqDist}
 Suppose $f$ satisfies \eqref{eq:Af}. Let $0 < \tau \leq T$, $z_1,z_2 \in C_0([0,\tau])$ 
 and $0 \leq M_+^i,M_-^i < \infty$, $i = 1,2$. Let $\rho_{z_i}^\pm$ be 
 defined with respect to $M^i_\pm$, $i = 1,2$, cf.~\eqref{eq:TruncRunningMaxMin}. 
 Then, for any $\alpha_1, \alpha_2 \in \mathscr{A}_\tau$, define
 \begin{multline*}
   \Phi[z_1,z_2](\alpha_1,\alpha_2)  \\ 
   :=\max\left(
      \begin{split}
	  &\seq{\abs{\rho_{z_1}^+ \circ \alpha_1(t)-\rho_{z_2}^+ \circ \alpha_2(t)}\,:\, 
	  t \in \mathcal{T}_{z_1 \circ \alpha_1}^+ \cup \mathcal{T}_{z_2 \circ \alpha_2}^+ 
	  \cap (\kappa^+,\tau]} \\ 
	  &\cup \seq{\abs{\rho_{z_1}^- \circ \alpha_1(t)-\rho_{z_2}^- \circ \alpha_2(t)}\,:\, 
	  t \in \mathcal{T}_{z_1 \circ \alpha_1}^- \cup \mathcal{T}_{z_2 \circ \alpha_2}^- 
	  \cap (\kappa^-,\tau]}\\ 
	  &\cup \seq{\abs{z_1(\tau)-z_2(\tau)}}
      \end{split}
	\right),
 \end{multline*}
 where $\kappa^\pm = \kappa^\pm_{z_1,z_2}(\alpha_1,\alpha_2)$ are defined by
 \begin{align*}
  \kappa^+ &= \inf\seq{t \in [0,\tau]\,:\,\max \seq{\rho_{z_1}^+ 
  \circ \alpha_1(t),\rho_{z_2}^+\circ \alpha_2(t)} > \max\seq{\frac{1}{M^1_-},\frac{1}{M^2_-}}}, \\
  \kappa^- &= \inf\seq{t \in [0,\tau]\,:\,\min \seq{\rho_{z_1}^- 
  \circ \alpha_1(t),\rho_{z_2}^-\circ \alpha_2(t)} < \min\seq{-\frac{1}{M^1_+},-\frac{1}{M^2_+}}},
 \end{align*}
 with the convention that $\kappa^\pm = \infty$ if the set is empty. 
 Set $\iota(t) = t$. Then
   \begin{align*}
    \norm{[z_1]_{\tau,M^1_\pm}-[z_2]_{\tau,M^2_\pm}}_\infty
	&= \inf\seq{\Phi[\tilde{z}_1,\tilde{z}_2](\iota,\iota)\,:\tilde{z}_1 
	\overset{\tau,M^1_\pm}{\sim} z_1, \tilde{z}_2 \overset{\tau,M^2_\pm}{\sim} z_2} \\
	&\leq \inf\seq{\Phi[\tilde{z}_1,\tilde{z}_2](\iota,\iota)\,:\tilde{z}_1 
	\overset{\tau,M_\pm^1}{\sim_\circ} z_1, \tilde{z}_2 \overset{\tau,M_\pm^2}{\sim_\circ} z_2} \\
	&= \inf\seq{\Phi[z_1,z_2](\alpha_1,\alpha_2)\,:\alpha_1,\alpha_2 \in \mathscr{A}_\tau}. 
   \end{align*}
 \end{theorem}
  Let us give a geometrical interpretation of the minimization problem
  \begin{displaymath}
   \inf\seq{\Phi[z_1,z_2](\alpha_1,\alpha_2)\,:\alpha_1,\alpha_2 \in \mathscr{A}_\tau}.
  \end{displaymath}
  To this end, let $\alpha:[0,\tau] \rightarrow [0,\tau]^2$ be the 
  parameterized path $\alpha(t) := (\alpha_1(t),\alpha_2(t))$. 
  Given a subset $\mathscr{S} \subset [0,\tau]^2$, denote by 
  $\mathscr{T}_\alpha(\mathscr{S})$ the first time $\alpha$ hits $\mathscr{S}$:
  \begin{displaymath}
  	\mathscr{T}_\alpha(\mathscr{S}) = 
	\inf\seq{t \in [0,\tau]\,:\, \alpha(t) \in \mathscr{S}}.
  \end{displaymath}
  Set $\ell_{\seq{s_i = t}} := \seq{(s_1,s_2) \in [0,\tau]^2\,:\,s_i = t}$, for $i = 1,2$. 
  By Lemma~\ref{lemma:OrmCompEquiv} and the continuity of $\alpha$,
  \begin{align*}
   \mathcal{T}_{z_i \circ \alpha_i}^\pm = \seq{\alpha_i^{-1}(t)\,:\,t \in \mathcal{T}_{z_i}^\pm}
	&= \seq{\inf\seq{s \in [0,\tau]\,:\, \alpha_i(s) \geq t}\,:\,t \in \mathcal{T}_{z_i}^\pm} \\
	&= \seq{\mathscr{T}_\alpha(\ell_{\seq{s_i = t}})\,:\,t \in \mathcal{T}_{z_i}^\pm},
  \end{align*}
  for $i = 1,2$. Similarly, $\kappa^\pm = \mathscr{T}_\alpha(K^\pm)$, where
  \begin{align*}
   K^+ &=
   \seq{(s_1,s_2) \in [0,\tau]^2\,:\,\max 
   \seq{\rho_{z_1}^+(s_1),\rho_{z_2}^+(s_2)} > \max\seq{\frac{1}{M^1_-},\frac{1}{M^2_-}}}, \\
   K^- &= 
   \seq{(s_1,s_2) \in [0,\tau]^2\,:\,\min 
   \seq{\rho_{z_1}^-(s_1),\rho_{z_2}^-(s_2)} < \min\seq{-\frac{1}{M^1_+},-\frac{1}{M^2_+}}}.
  \end{align*}
  Consequently,
  \begin{multline*}
   \mathcal{T}_{z_1 \circ \alpha_1}^\pm 
   \cup \mathcal{T}_{z_2 \circ \alpha_2}^\pm \cap (\kappa^\pm,\tau]
    = \seq{\mathscr{T}_\alpha(\ell_{s_1 = t} \cap K^\pm) \,:\, t \in \mathcal{T}_{z_1}^\pm} \\
      \cup \seq{\mathscr{T}_\alpha(\ell_{s_2 = t} \cap K^\pm) \,:\, t \in \mathcal{T}_{z_2}^\pm}.
  \end{multline*}
   In other words, the value $\Phi[z_1,z_2](\alpha_1,\alpha_2)$ is 
   dependent only on where the path $\alpha$ hits 
   $\mathcal{L^+} \cup \mathcal{L^-}$; $\mathcal{L^\pm}$ are 
   the line segments
   \begin{displaymath}
    \mathcal{L}^\pm := \seq{\ell_{s_1 = t} \cap K^\pm \,:\, 
    t \in \mathcal{T}_{z_1}^\pm} \cup 
    \seq{\ell_{s_2 = t} \cap K^\pm \,:\, t \in \mathcal{T}_{z_2}^\pm},
   \end{displaymath}
   see Figure~\ref{fig:AlphaMin}. As a result, 
   $\Phi[z_1,z_2](\alpha_1,\alpha_2)$ is a function of the path $\alpha$, 
   independent of its parameterization. 

From the view of factoring the solution operator, Theorem~\ref{thm:EqDist} 
is supplying a description of the metric induced by the 
uniform norm on the quotient space, cf.~Figure~\ref{fig:Factorization}.

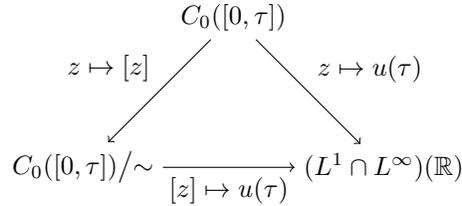
\begin{figure}[h]
\begin{tikzpicture}[node distance=2cm, auto]
  \node (C) {$C_0([0,\tau])$};
  \node (Q) [below of=C, left of=C] {$C_0([0,\tau])\big/{\sim}$};
  \node (S) [below of=C, right of=C] {$(L^1 \cap L^\infty)(\R)$};
  \draw[->] (C) to node {$z \mapsto u(\tau)$} (S);
  \draw[->] (C) to node [swap] {$z \mapsto [z]$} (Q);
  \draw[->] (Q) to node [swap] {$[z] \mapsto u(\tau)$} (S);
\end{tikzpicture}
\caption{Schematic drawing of the factorization of the solution map.}
\label{fig:Factorization}
\end{figure}

   Based on the above observations, we now give an outline of 
   how the minimization problem may be solved using dynamic programming.  
   Introduce a cost function $c:[0,\tau]^2 \rightarrow [0,\infty)$ by
   \begin{displaymath}
   	 c(s) =  \max\seq{\begin{split}             
		\abs{\rho_{z_1}^+(s_1)-\rho_{z_2}^+(s_2)} \1{\mathcal{L}^+}(s),\\
		\abs{\rho_{z_1}^-(s_1)-\rho_{z_2}^-(s_2)}\1{\mathcal{L}^-}(s), \\
		\abs{z_1(\tau)-z_2(\tau)}\1{\seq{s = (\tau,\tau)}}
		\end{split}},
   \end{displaymath}
  where $\1{\mathscr{S}}$ is the characteristic 
  function of $\mathscr{S}$. Hence, 
  \begin{displaymath}
   \Phi[z_1,z_2](\alpha_1,\alpha_2) = 
   \max\seq{c(\alpha(t))\,:\,t \in \cup_{i = 1,2}
   \seq{\mathscr{T}_\alpha(\ell_{s_i = t}) \,:\, t \in \mathcal{T}_{z_i}} \cup \seq{\tau}}.
  \end{displaymath}
  For any $s \in [0,\tau]^2$, let $\mathcal{A}_s$ be the set of 
  monotone paths connecting $s$ and $(\tau,\tau)$: 
\begin{displaymath}
   \mathcal{A}_s := \seq{\alpha \in C([0,1];[0,\tau]^2)\,:\, 
   \begin{split} 
   	&\alpha(0) = s, \alpha(1) = (\tau,\tau), \alpha = (\alpha_1,\alpha_2), 
   	 \mbox{ where }\\ 
    	&\alpha_1,\alpha_2 \mbox{ are nondecreasing and continuous.}
   \end{split}}.
\end{displaymath}
Define a value function $V:[0,\tau]^2 \rightarrow [0,\infty)$ by
   \begin{align*}
    V(s) = \inf_{\alpha \in \mathcal{A}_s}
    \seq{\max\seq{c(\alpha(t))\,:\,t \in \cup_{i = 1,2}
    \seq{\mathscr{T}_\alpha(\ell_{s_i = \tilde t}) 
    \,:\, \tilde t \in \mathcal{T}_{z_i}} \cup \seq{1}}},
   \end{align*}
   so that
   \begin{displaymath}
    V((0,0)) = \inf\seq{\Phi[z_1,z_2](\alpha_1,\alpha_2)\,: 
    \alpha_1,\alpha_2 \in \mathscr{A}_\tau}
   \end{displaymath}
   is the sought value. 
   Let us show how to compute $V$ on the grid 
   $$
   \mathcal{G} = \left(\cup_{j} \ell_{s_1 = \tau_j^1}\right) 
   \cup \left(\cup_k \ell_{s_2 = \tau_k^2}\right),
   $$
   where $\seq{\tau^i_n}_{n = 0}^{N_i} = \mathcal{T}_{z_i}$. 
   First note that for $s \in \ell_{s_1 = \tau}$, the set 
   of admissible paths $\mathcal{A}_s$ is simply any path 
   tracing out the straight line connecting $s$ and $(\tau,\tau)$. Hence, 
   \begin{displaymath}
    V((\tau,s_2)) = \max\seq{c(\tau,t)\,:\, t \in (\seq{s_2} 
    \cup \mathcal{T}_{z_2}) \cap [s_2,\tau]}.
   \end{displaymath}
   Similarly,
   \begin{displaymath}
    V((s_1,\tau)) = \max\seq{c(t,\tau)\,:\, 
    t \in (\seq{s_1} \cup \mathcal{T}_{z_1}) \cap [s_1,\tau]}.
   \end{displaymath} 
   To compute $V$ on the remaining part of $\mathcal{G}$, define the squares 
   $$
   Q_{j,k} = [\tau_j^1,\tau_{j-1}^1] \times [\tau_k^2,\tau_{k-1}^2]
   $$ and let
   \begin{align*}
    \partial^-Q_{j,k} &= \seq{\tau^1_j} \times [\tau^2_k,\tau^2_{k-1}] 
    \cup [\tau^1_j,\tau^1_{j-1}] \times \seq{\tau^2_k}, \\
    \partial^+Q_{j,k} &= \seq{\tau^1_{j-1}} 
    \times [\tau^2_k,\tau^2_{k-1}] \cup [\tau^1_j,\tau^1_{j-1}] \times \seq{\tau^2_{k-1}}.
   \end{align*}
   Suppose $V$ is known on $\partial^+Q_{j,k}$. Then, for $s \in \partial^-Q_{j,k}$,
   \begin{displaymath}
    V(s) = \max\seq{c(s), \min\seq{V(\tilde{s})\,:\,\tilde s \in \partial^+Q_{j,k} 
    \cap \seq{\bar s \in [0,\tau]^2 \,:\,\bar s \geq s}}}.
   \end{displaymath}
   As $\partial^+Q_{j,k} \subset \partial^-Q_{j-1,k} 
   \cup \partial^-Q_{j,k-1}$ for $j,k > 1$  we may compute $V$ on 
   the entire grid $\mathcal{G}$, starting in the upper right 
   square $Q_{1,1}$ and trace our way down to the 
   lower left square $Q_{N_1,N_2}$.
\begin{figure}[h]
\begin{tikzpicture}

 \node at (3,-0.5) {$\mathcal{T}_{z_1}^+ \cup \mathcal{T}_{z_1}^-$};
 \node[rotate = 90] at (-0.5,3){$\mathcal{T}_{z_2}^+ \cup \mathcal{T}_{z_2}^-$};
 
 \draw[dashed] (5.5,-0.1) -- (5.5,5.5);
 \node[below=6pt] at (5.5,0.0){$\tau$};
 \draw[dashed] (-0.1,5.5) -- (5.5,5.5);
 \node[left=6pt] at (0,5.5){$\tau$};
 \draw (4.1,-0.1) -- (4.1,5.5);
 \draw (-0.1,5.1) -- (5.5,5.1);
 \draw (-0.1,2.3) -- (5.5,2.3);
 \draw (1.5,-0.1) -- (1.5,5.5);
 %Complement of K^+,K^- regions
 \fill[color=white] (0,0) rectangle (3.2,1.3); %To shade over line
 \fill[pattern=north east lines, pattern color=gray] (0,0) rectangle (1.2,2.5);
 \fill[pattern=north west lines, pattern color=gray] (0,0) rectangle (3.2,1.3);

 \draw (2.1,-0.1) -- (2.1,5.5);
 %Draw alpha
 \draw[thick] (0,0) -- (2.1,1.0) -- (2.3,2.3) -- (4.1,4.9) -- (5.5,5.1) -- (5.5,5.5);

 \draw[thick, ->] (0,0) -- (6,0) node[right]{$s_1$};
 \draw[thick, ->] (0,0) -- (0,6) node[above]{$s_2$};
\end{tikzpicture}
\caption{The curve $\alpha = (\alpha_1,\alpha_2)$ along 
with the line segments in $\mathcal{L}^+$ and $\mathcal{L}^-$. 
The complement of $K^+$ and $K^-$ are hatched in grey.}
\label{fig:AlphaMin}
\end{figure}
 
\section{Proofs of main results}\label{seq:Proofs}
In this section we provide detailed proofs of 
Theorems \ref{thm:OrmEquiv}, \ref{thm:OSB}, \ref{thm:GeneralEquiv}, 
and \ref{thm:EqDist}.

\subsection{Local solutions by substitution}
For a given path $z \in C^1([0,T])$, let $v$ solve the conservation law
\begin{equation*}
 \partial_tv + \partial_xf(v) = 0,  \qquad v(z(0),x) = u_0(x)
\end{equation*}
on $[z(0),\infty)\times \R$, and set $u(t,x) := v(z(t),x)$. Then, formally, it follows 
that $u$ solves \eqref{eq:sscl}. Let us take a closer look at this 
substitution, by considering the viscous approximation. That is, let $v^\varepsilon$ 
be the classical solution to the parabolic problem
\begin{equation*}
 \partial_t\ve + \partial_x f(\ve) = 
 \varepsilon \partial_x^2 \ve, \qquad \ve(z(0),x) = u_0(x)
\end{equation*}
on $[z(0),\infty)\times \R$. Then $\ue(t,x) := \ve(z(t),x)$ satisfies, for any 
convex entropy, entropy-flux pair $(S,Q)$,
\begin{equation*}
 \partial_tS(\ue) + \partial_xQ(\ue)\dot{z} = 
 \varepsilon\partial_x^2S(\ue)\dot{z} 
 -\varepsilon S''(\ue)(\partial_x\ue)^2\dot{z}, \quad \ue(0,x) = u_0(x).
\end{equation*}
A priori, due to the factor $\dot{z}$, the limiting 
solution does not necessarily dissipate the entropy. However, if 
$\partial_xv(z(t),x) = \partial_xu(t,x)$ is bounded, then 
the dissipation vanishes as $\varepsilon\to 0$ 
and $u$ ought to be a solution. Also, if $\dot{z} \geq 0$, 
then $\ue$ ought to converge to the entropy solution to \eqref{eq:sscl}.  
Similarly, if $\dot{z} \le 0$, we let $\ve$ solve 
the parabolic problem with flux $-f$,
\begin{equation*}
 \partial_t\ve - \partial_x f(\ve) 
 = \varepsilon \partial_x^2 \ve, 
 \qquad \ve(-z(0),x) = u_0(x)
 \end{equation*}
on $(-\infty,-z(0))\times \R$, and take $\ue(t,x) = \ue(-z(t),x)$. 

These observations are formalized in the next two lemmas.
We consider first the case that $z$ is monotone.
\begin{lemma}\label{lemma:EntSub}
  Assume that $f$ is Lipschitz continuous and $u_0 \in
  (L^\infty \cap L^1 \cap BV)(\R)$. Suppose $z$ is Lipschitz
  continuous and nondecreasing on $[0,\tau]$ and $\tilde{z}$ is
  Lipschitz continuous on $[z(0),z(\tau)]$. Let $v$ be the entropy solution to
  \begin{equation*}
    \begin{cases}
      \partial_tv + \tilde{z}'\partial_xf(v) = 0, 	&(t,x) \in
      (z(0),z(\tau)) \times \R,\\ 
      v(z(0),x) = u_0(x) &x \in \R.
    \end{cases}
  \end{equation*}
  Set $u(t,x) := v(z(t),x)$. Then $u$ is an entropy solution to
  \begin{equation*}
    \begin{cases}
      \partial_tu + (\tilde{z} \circ z)'\partial_x f(u) = 0, 	&(t,x) \in (0,\tau) \times \R, \\
      u(0,x) = u_0(x), &x \in \R.
    \end{cases}
  \end{equation*}
\end{lemma}
\begin{proof}
  Assume first that $z$ is strictly increasing, i.e., $z' > 0$ on $(0,\tau)$, 
  and set $z(0) = a$, $z(\tau) = b$. We need to
  show that $u(t,x) = v(z(t),x)$ satisfies
  \begin{equation}\label{eq:EntIn}
    \int_0^\tau \int_\R S(u)\partial_t \test  + (\tilde{z} \circ
    z)'Q(u)\partial_x \test \,dxdt  
    + \int_\R S(u_0(x))\test(0,x)\,dx \geq 0,
  \end{equation}
  for all convex entropy, entropy-flux pairs $(S,Q)$ and for all
  non-negative test function $\test \in C^\infty_c([0,\tau) \times \R)$. 
  Let $\phi(\zeta,x) = \test(z^{-1}(\zeta),x)$.  Upon a change of
  variables it follows that
  \begin{align*}
    \int_0^\tau \int_\R S(u)\partial_t \test \,dxdt
    &= \int_0^\tau \int_\R S(v(z(t),x))\partial_z\phi(z(t),x)\dot{z}(t)\,dxdt \\
    &= \int_{a}^{b} \int_\R S(v(z,x))\partial_z\phi(z,x)\,dxdz.
  \end{align*}
  Similarly,
  \begin{equation*}
    \int_0^\tau \int_\R (\tilde{z} \circ z)'(t)Q(u)\partial_x \test \,dxdt 
    = \int_{a}^{b}\int_\R \tilde{z}'(z)Q(v(z,x))\partial_x \phi(z,x) \,dxdz,
  \end{equation*}
  and
  \begin{equation*}
    \int_\R S(u_0(x))\test(0,x)\,dx = \int_\R S(v(a,x))\phi(a,x)\,dx.
  \end{equation*}
  Hence, \eqref{eq:EntIn} follows due to the fact that $v$ is an
  entropy solution. 
  
  Next, suppose $z$ is merely nondecreasing, i.e., $z'\ge 0$ on $(0,\tau)$, and 
  $z(\tau) > z(0)$. Introduce the approximation 
  \begin{equation*}
    z_\varepsilon(t) = (1-\varepsilon)z(t) + \varepsilon\left[\frac{z(\tau)-z(0)}{\tau}t + z(0)\right], 
    \qquad 0 < \varepsilon < 1,
  \end{equation*} 
  and note that $z_\varepsilon' > 0$. 
  Take $u^\varepsilon(t,x) = v(z_\varepsilon(t),x)$ and send $\varepsilon
  \downarrow 0$ in \eqref{eq:EntIn}. 
  
  If $z' \equiv 0$ on $(0,\tau)$, then
  $\tilde{z}(t) = z(0)$ and $u(t,x) = v(z(t),x) = u_0(x)$, so
  \eqref{eq:EntIn} is satisfied.
\end{proof}
In the following discussion it will be convenient for us to talk about 
\emph{backward entropy solutions}. 
For us the natural backward solution is the (forward) 
entropy solution to the problem with flux $-f$. 
That is, the entropy solution to 
\begin{equation}\label{eq:TimeHomNeg}
 \partial_tv + \partial_xf(v) = 0, \qquad v(0,x) = v_0(x)
\end{equation}
on $(-\infty,0]\times \R$ is obtained by solving the problem 
\begin{equation*}
 \partial_tw - \partial_xf(w) = 0, \qquad w(0,x) = v_0(x),
\end{equation*}
on $[0,\infty)\times \R$ and setting $v(t,x) = w(-t,x)$. 
This yields the following backward Kru\v{z}kov entropy condition 
for \eqref{eq:TimeHomNeg} on the interval $(-\infty,0]$. 
For any convex entropy, entrop-flux pair $(S,Q)$, and for any 
non-negative $\test \in C^\infty_c((-\infty,0] \times \R)$,
\begin{equation*}
   \int^0_{-\infty}\int_\R S(v)\partial_t\test + Q(v)\partial_x\test\,dxdt 
   - \int_\R S(v_0(x))\test(0,x)\,dx \le 0.
\end{equation*}
By the one-to-one correspondence $v \mapsto w$ it is clear 
that the associated notion of \emph{backward entropy solution} is well posed for 
any $u_0 \in (L^\infty \cap L^1)(\R)$. 
The \emph{backward/forward} entropy solution 
to \eqref{eq:TimeHomNeg} on $(-\infty,\infty) \times \R$ is 
obtained by considering the forward solution for $t \geq 0$ and 
the backward solution for $t < 0$. The fact that the initial 
condition was specified at time $t_0 = 0$ 
was somehow arbitrary, and the extension to general $t_0 \in \R$ 
may be obtained by the substitution $t \mapsto t-t_0$.  

\begin{lemma}\label{lemma:VanishingDissip}
 Suppose $u_0 \in (L^\infty \cap L^1 \cap BV)(\R)$ and $f$ is 
 Lipschitz continuous. Let $z$ be a Lipschitz continuous path on $[0,\tau]$. Write 
 \begin{equation*}
  z_\mathrm{min} = \min_{0 \le t \le \tau}\seq{z(t)} \quad \text{and} \quad 
  z_\mathrm{max} = \max_{0 \le t \le \tau}\seq{z(t)}.
 \end{equation*}
 Let $u$ be the entropy solution to $\eqref{eq:sscl}$ on $[0,\tau]$ with initial 
 condition $u_0$, and $v$ be the backward/forward entropy solution to 
 \begin{equation}\label{eq:TimeHom1}
  \partial_tv + \partial_xf(v) = 0, \qquad v(z(0),x) = u_0(x),
 \end{equation}
 on $[z_\mathrm{min},z_\mathrm{max}]$. Suppose $v$ is uniformly Lipschitz
 continuous in $x$ for any $t\in [t_1,t_2] \subset
 (z_\mathrm{min},z_\mathrm{max})$. Then 
 $u(t,x) =v(z(t),x)$. Furthermore $u$ is nondissipative, i.e.,
for any convex entropy, entropy-flux pair $(S,Q)$ and for any 
$\test \in C^\infty_c([0,\tau) \times \R)$,
 \begin{equation}\label{eq:NondissipEnt}
    \int_0^\tau \int_\R S(u)\partial_t \test  + \dot{z}Q(u)\partial_x \test \,dxdt 
    + \int_\R S(u_0(x))\test(0,x)\,dx = 0.
 \end{equation}
\end{lemma}
\begin{proof}
  By the weak formulation of \eqref{eq:TimeHom1}, using the Lipschitz
  continuity in $x$, for any $[t_1,t_2]
  \subset (z_\mathrm{min},z_\mathrm{max})$,
  \begin{equation*}
    \abs{v(t_2,x)-v(t_1,x)} \le \sup_{t_1 \le t \le t_2}
    \seq{\norm{\partial_xf(v)}_{L^\infty(\R)}}\abs{t_2-t_1}.
  \end{equation*}
  Consequently $v$ is locally Lipschitz continuous in time. Let us
  consider an approximation $\seq{z_\varepsilon}_{\varepsilon > 0}
  \subset C^1([0,\tau])$ of $z$ satisfying
  \begin{equation*}
    z_\varepsilon \rightarrow z \mbox{ uniformly 
    on $[0,\tau]$}, \qquad z_\mathrm{min} +\varepsilon 
    \le z_\varepsilon \le z_\mathrm{max} -\varepsilon.
  \end{equation*}
  Let $\ue(t,x) = v(z_\varepsilon(t),x)$. As $t \mapsto v(t)$ is
  $L^1$-Lipschitz continuous, $u_\varepsilon \rightarrow u$ in
  $C([0,\tau];L^1(\R))$. Using integration by parts and the chain rule
  \cite[Theorem~3.99]{AmbrosioFuscoPallara2000} we obtain
  \begin{align*}
    \int_0^\tau & \int_\R S(u)\partial_t \test  +
    \dot{z}Q(u)\partial_x \test \,dxdt + \int_\R
    S(u_0(x))\test(0,x)\,dx \\ 
    &= \lim_{\varepsilon \downarrow 0} \int_0^\tau \int_\R
    S(\ue)\partial_t \test  + \dot{z_\varepsilon}Q(\ue)\partial_x
    \test \,dxdt + \int_\R  S(v(z_\varepsilon(0),x))\test(0,x)\,dx \\
    &= -\lim_{\varepsilon \downarrow 0} \int_0^\tau \int_\R
    \underbrace{\left[\partial_z S(v(z_\varepsilon(t),x))
        + \partial_xQ(v(z_\varepsilon(t),x))\right]}_{m_v(z_\varepsilon(t),x)}
    \dot{z_\varepsilon}(t)\test(t,x) 
    \,dxdt.
  \end{align*}
  By Lipschitz continuity, $m_v \in
  L^\infty([z_\mathrm{min}+\varepsilon,z_\mathrm{max} - \varepsilon]
  \times \R)$ for $\varepsilon > 0$. Moreover, by the
  chain-rule, $m_v(z,x) = 0$ for almost all $(z,x) \in
  [z_\mathrm{min},z_\mathrm{max}] \times \R$. Hence, we conclude that
  \eqref{eq:NondissipEnt} holds.
\end{proof}

\subsection{Spatial regularity estimates for piecewise linear paths.}
Let $u$ be the entropy solution to \eqref{eq:sscl} with path 
$z$. In view of Lemma~\ref{lemma:VanishingDissip}, it is necessary to keep track 
of the best possible bounds $t \mapsto (-M_-(t),M_+(t))$ satisfying 
\begin{equation*}
	-M_-(t) \le \frac{f'(u(t,y))-f'(u(t,x))}{y-x} \le M_+(t),
	\quad x,y\in \R, \, x<y.
\end{equation*}

We begin with the following Ole{\u\i}nik-type estimate:

\begin{lemma}\label{lemma:OneSidedBounds}
 Suppose \eqref{eq:Af} holds and $u_0 \in L^\infty(\R) \cap BV(\R)$ 
 satisfies \eqref{eq:LowUpOnInitial}. Fix $\lambda \in \R$ 
 and let $u$ be the entropy solution to
 \begin{equation*}
  \partial_t u + \lambda \partial_xf(u) = 0, \qquad u(0,x) = u_0(x),
 \end{equation*}
 for some constant $\lambda$.
 Let $\Theta(x) = x$ for $x \geq 0$ and $\Theta(x) = \infty$ for $x < 0$. Then
 \begin{equation}\label{eq:UpLowBound}
  -\Theta\left(\frac{1}{(M_-)^{-1}-t\lambda}\right) \le \frac{f'(u(t,y\pm))-f'(u(t,x\pm))}{y-x} 
  \le \Theta\left(\frac{1}{(M_+)^{-1} + t\lambda}\right)
 \end{equation}
 for all $-\infty < x < y < \infty$. Whenever $M_\pm$ takes the values $0$ or $\infty$, we 
 use the convention $(\infty)^{-1} = 0$ and $(0)^{-1} = \infty$.
\end{lemma}

\begin{proof}
  Let $\xi_\pm$ be the maximal/minimal backward characteristic
  emanating from $(t,x)$ and $\zeta_\pm$ be the maximal/minimal
  backward characteristic emanating from $(t,y)$ on
  $[0,t]$,~cf.~\cite[\S~10.2]{Dafermos:2010fk}. By
  \cite[Theorem~10.3.2]{Dafermos:2010fk}, both $\xi_\pm$ and
  $\zeta_\pm$ are shock free on $[\tau,t]$. In view of
  \cite[Theorem~11.1.1]{Dafermos:2010fk},
  \begin{equation*}
    x = \xi_\pm(0) + t\lambda f'(u(t,x\pm)), \quad \mbox{ and } y =
    \zeta_\pm(0) + t\lambda f'(u(t,y\pm)). 
  \end{equation*}
  where $\xi_\pm(0) \le \zeta_\pm(0)$. Note, these equalities are true
  when we chose either $+$ or $-$. That is, when considering $x+$, we
  apply $\xi_+$ and so forth. The case $\lambda = 0$ is trivial. We
  consider the cases $\lambda > 0$ and $\lambda < 0$ separately.

  \emph{Assume $\lambda > 0$}. Consider the upper bound in
  \eqref{eq:UpLowBound}. Assume $f'(u(t,y\pm))-f'(u(t,x\pm)) > 0$. It
  follows that
  \begin{equation}\label{eq:InvRegChar}
    \frac{y-x}{f'(u(t,y\pm))-f'(u(t,x\pm))} =
    \frac{\zeta_\pm(0)-\xi_\pm(0)}{f'(u(t,y\pm))-f'(u(t,x\pm))} +
    t\lambda. 
  \end{equation}
  By \cite[Theorem~11.1.3]{Dafermos:2010fk},
  \begin{equation*}
    \begin{cases}
      u_0(\zeta_\pm(0)-) \le u(t,y\pm) \le u_0(\zeta_\pm(0)+), \\
      u_0(\xi_\pm(0)-) \le u(t,x\pm) \le u_0(\xi_\pm(0)+).
    \end{cases}
  \end{equation*}
  Hence, since $f'$ is increasing,
  \begin{equation*}
    0 < f'(u(t,y\pm))-f'(u(t,x\pm)) \le f'(u_0(\zeta_\pm(0)+))-f'(u_0(\xi_\pm(0)-)).
  \end{equation*}
  If $M_+ = 0$ this cannot be true, and so
  $f'(u(t,y\pm))-f'(u(t,x\pm)) \le 0$, which proves the upper bound
  for $M_+ = 0$. Assume $M_+ > 0$. Then, by the above,
  \begin{align*}
    \frac{y-x}{f'(u(t,y\pm))-f'(u(t,x\pm)}
    &\geq \frac{\zeta_\pm(0)-\xi_\pm(0)}{f'(u_0(\zeta_\pm(0)+)))-f'(u_0(\xi_\pm(0)-))} 
    + t\lambda \\
    &\geq \frac{1}{M_+} + t\lambda.
  \end{align*}
  Consider the lower bound in \eqref{eq:UpLowBound}. Assume
  $f'(u(t,y\pm))-f'(u(t,x\pm)) < 0$. Arguing as above,
  \begin{equation}\label{eq:CharInitIneqNeg}
    0 > f'(u(t,y\pm))-f'(u(t,x\pm)) \geq f'(u_0(\zeta_\pm(0)-))-f'(u_0(\xi_\pm(0)+)).
  \end{equation}
  If $M_- = 0$, this yields a contradiction, and so
  $f'(u(t,y\pm))-f'(u(t,x\pm)) \geq 0$, proving the lower bound for
  $M_- = 0$. Assume $M_- > 0$. Inserting \eqref{eq:CharInitIneqNeg}
  into \eqref{eq:InvRegChar} yields
  \begin{align*}
    \frac{y-x}{f'(u(t,y\pm))-f'(u(t,x\pm)}
    &\le \frac{\zeta_\pm(0)-\xi_\pm(0)}{f'(u_0(\zeta_\pm(0)+)))-f'(u_0(\xi_\pm(0)-))} 
    + t\lambda \\
    &\le -\frac{1}{M_-} + t\lambda.
  \end{align*}
  Assuming $t\lambda < (M_-)^{-1}$ yields the lower bound in
  \eqref{eq:UpLowBound}.

  \emph{Assume $\lambda < 0$}. By \cite[Theorem~11.1.3]{Dafermos:2010fk}
  (upon reversing the inequality as $\lambda f$ is concave) we have
  \begin{equation}\label{eq:CharValueConcave}
    \begin{cases}
      u_0(\zeta_\pm(0)+) \le u(t,y\pm) \le u_0(\zeta_\pm(0)-), \\
      u_0(\xi_\pm(0+) \le u(t,x\pm) \le u_0(\xi_\pm(0)-).
    \end{cases}
  \end{equation}
  Consider the upper bound in \eqref{eq:UpLowBound}. Assume
  $f'(u(t,y\pm))-f'(u(t,x\pm)) > 0$. Then
  \begin{equation*}
    0 < f'(u(t,y\pm))-f'(u(t,x\pm)) \le f'(u_0(\zeta_\pm(0)-))-f'(u_0(\xi_\pm(0)+)).
  \end{equation*}
  If $M_+ = 0$ we obtain a contradiction. Assume $M_+ > 0$. Inserting
  into \eqref{eq:InvRegChar} yields
  \begin{align*}
    \frac{y-x}{f'(u(t,y\pm))-f'(u(t,x\pm)}
    &\geq \frac{\zeta_\pm(0)-\xi_\pm(0)}{f'(u_0(\zeta_\pm(0)-)))-f'(u_0(\xi_\pm(0)+))} + t\lambda \\
    &\geq \frac{1}{M_+} + t\lambda.
  \end{align*}
  Assuming $-t \lambda < (M_+)^{-1}$ yields the upper bound in
  \eqref{eq:UpLowBound}. Consider the lower bound in
  \eqref{eq:UpLowBound}. Assume $f'(u(t,y\pm))-f'(u(t,x\pm)) < 0$. By
  \eqref{eq:CharValueConcave},
  \begin{equation*}
    0 > f'(u(t,y\pm))-f'(u(t,x\pm)) \geq f'(u_0(\zeta_\pm(0)+))-f'(u_0(\xi_\pm(0)-)).
  \end{equation*}
  If $M_- = 0$ we obtain a contradiction. Assume $M_- > 0$. Inserting
  into \eqref{eq:InvRegChar} yields
  \begin{align*}
    \frac{y-x}{f'(u(t,y\pm))-f'(u(t,x\pm)}
    &\le \frac{\zeta_\pm(0)-\xi_\pm(0)}{f'(u_0(\zeta_\pm(0)+))-f'(u_0(\xi_\pm(0)-))} + t\lambda \\
    &\le -\frac{1}{M_-} + t\lambda.
  \end{align*}
  This yields the lower bound for \eqref{eq:UpLowBound}.
\end{proof}

Combining Lemmas \ref{lemma:VanishingDissip} 
and \ref{lemma:OneSidedBounds} we obtain

\begin{corollary}\label{cor:RegInitPathSub}
  Suppose \eqref{eq:Af} holds, and that $u_0 \in (L^\infty
  \cap L^1 \cap BV)(\R)$ satisfies \eqref{eq:LowUpOnInitial} 
  for some $0 \le M_-,M_+ \le \infty$. Let $z$ be a Lipschitz continuous path, satisfying
  \begin{equation}\label{eq:PathPieceCrit}
    z(0) = 0, \qquad -(M_+)^{-1} \le z(t) \le (M_-)^{-1},
    \qquad t \in [0,\tau],
  \end{equation}
  for some $\tau > 0$. 
  Let $v$ be the
  backward/forward entropy solution to
  \begin{equation}\label{eq:TimeHom}
    \partial_t v + \partial_x f(v) = 0, \qquad v(0,x) = u_0(x),
  \end{equation}
  on $[-(M_+)^{-1},(M_-)^{-1}]$. We apply the convention
  $(0)^{-1} = \infty$ and $(\infty)^{-1} = 0$.  Then $u(t,x) :=
  v(z(t),x)$ is a nondissipative entropy solution to \eqref{eq:sscl}
  on $[0,\tau]$.
\end{corollary}
\begin{proof}
  Recall that $v$ is composed of a backward and a forward
  solution. That is,
  \begin{equation*}
    v(t,x) :=
    \begin{cases}
      w_+(t,x) \,  	&t \in [0,(M_-)^{-1}), \\
      w_-(-t,x) \, &t \in (-(M_+)^{-1},0),
    \end{cases}
  \end{equation*}
  where $w_\pm$ are the entropy solutions to
  \begin{equation*}
    \begin{cases}
      \partial_tw_+ + \partial_xf(w_+) = 0, \quad t \in (0,(M_-)^{-1}), \qquad w_+(0,x) = u_0(x), \\
      \partial_tw_- - \partial_xf(w_-) = 0, \quad t \in
      (0,(M_+)^{-1}), \qquad w_-(0,x) = u_0(x).
    \end{cases}
  \end{equation*}
  Consequently, by Lemma~\ref{lemma:OneSidedBounds},
  \begin{equation*}
    -\frac{1}{(M_-)^{-1}-t} \le \frac{f'(v(t,y\pm))-f'(v(t,x\pm))}{y-x} \le \frac{1}{(M_+)^{-1} + t}
  \end{equation*}
The result follows from Lemma~\ref{lemma:VanishingDissip}.
\end{proof}
We are now ready to prove Theorem~\ref{thm:OSB} for 
piecewise linear paths. 
However, it will be convenient for us to apply 
Lemma~\ref{lemma:BpBmProp}, so we begin with its proof.

\begin{proof}[Proof of Lemma~\ref{lemma:BpBmProp}]
  (i). Suppose $t \in \Bp_z$, $t > 0$. The argument for $\Bm_z$
  is analogous. By definition $\rho_z^+(t) \ge z(t)$. To arrive at a
  contradiction, suppose $\rho_z^+(t) > z(t)$. By continuity of
  $\rho_z^+$, there exits $\delta > 0$ such that $\rho_z^+(s) >
  \rho_z^+(t)$ for all $0 < t-s < \delta$, contradicting the fact that
  $\inf \seq{s \in [0,T]\,:\, \rho_z^+(s) \ge \rho_z^+(t)} =
  t$. Suppose $t \in \Bm_z \cap \Bp_z$, $t > 0$. Then, by the above,
  $\rho_z^-(t) = z(t) = \rho_z^+(t)$, which can only be true if $z(t)
  = 0$. Consequently
  \begin{equation*}
    \inf \seq{s \in [0,T]\,:\, \rho_z^+(s) \ge \rho_z^+(t)} = 0,
  \end{equation*}
  contradicting the fact that $t \in \Bp_z \setminus \seq{0}$.

(ii).  Assume that $t_n \uparrow t$ where $\seq{t_n}_{n \ge 1}
  \subset \Bp_z$. We need to show that $t \in \Bp_z$.  As $t_n \in
  \Bp_z$
  \begin{equation*}
    \seq{s \in [0,T]\,:\, \rho_z^+(s) \ge \rho_z^+(t_n)} = [t_n,T],
  \end{equation*}
  for all $n \ge 1$. Consequently, as $\rho_z^+(t_n) \uparrow
  \rho_z^+(t)$,
  \begin{equation*}
    \inf\seq{s \in [0,T]\,:\, \rho_z^+(s) \ge \rho_z^+(t)} = \inf
    \Big(\bigcap_{n \ge 1}[t_n,T]\Big) = t, 
  \end{equation*}
  showing that $t \in \Bp_z$. For the second statement, let $t =
  \sup\seq{\Bpm_z \cap [0,\tau]}$. By definition of the supremum, there
  exists an increasing sequence $\seq{t_n}_{n \ge 1} \subset \Bp_z
  \cap [0,\tau]$ such that $t_n \uparrow t$.
 
(iii). Suppose $(t_1,t_2] \cap \Bp_z = \emptyset$. Then
  $\rho_z^+$ is constant on $[t_1,t_2]$. The proof for $\Bm_z$, $\rho_z^-$ is 
  analogous. Assume there exist $s_1$ and $s_2$ with $t_1 < s_1 < s_2 \le
  t_2$ such that $\rho_z^+(s_2) > \rho_z^+(s_1)$.  Then
  \begin{equation*}
    \rho_z^+(s_2)-\rho_z^+(s_1) > 2\varepsilon(s_2-s_1) \mbox{ for some $\varepsilon > 0$.}
  \end{equation*}
  Let
  \begin{equation*}
    t^* = \inf \seq{s \in [s_1,s_2] \,:\, \rho_z^+(s_2)-\rho_z^+(s) \le \varepsilon (s_2-s)}.
  \end{equation*}
  By continuity of $\rho_z^+$ it follows that $s_1 < t^* < s_2$. As
  $t^* \notin \Bp_z$ it follows by (i) that
  \begin{equation*}
    \inf \seq{s \in [0,T]\,:\, \rho_z^+(s) \ge \rho_z^+(t^*)} = \hat{t} < t^*.
  \end{equation*}
  Since $\rho^+$ is nondecreasing, $\rho^+(\hat{t})=\rho^+(t^*)$, and
  $\rho^+$ is constant on $[\hat{t},t^*]$. 
  But then, for any $\hat{t} \le s \le t^*$,
  \begin{equation*}
    \rho_z^+(s_2)-\rho_z^+(s) =  \rho_z^+(s_2)-\rho_z^+(t^*) =
    \varepsilon (s_2-t^*) \le \varepsilon (s_2-s), 
  \end{equation*}
  contradicting the definition of $t^*$.

(iv). We first prove the following claim: If
  $\partial_-\rho_z^\pm(t) = 0$ for all $t \in (t_1,t_2]$, then
  $\rho_z^\pm$ is constant on $[t_1,t_2]$. Let $s_1$, $s_2$, $t^*$ and
  $\varepsilon$ be as in the proof of (iii), and we also assume that
  $\rho^+(s_2)>\rho^+(s_1)$.  By assumption
  $\partial_-\rho_z^+(t^*) = 0$, so there exists a $\delta > 0$ such that
  \begin{equation*}
    \rho_z^+(t^*)-\rho_z^+(s) \le \varepsilon(t^*-s) 
    \quad \mbox{ for all } \quad t^*-\delta < s < t^*.
  \end{equation*}
  But then, for any $t^*-\delta < s < t^*$,
  \begin{equation*}
    \rho_z^+(s_2)-\rho_z^+(s) = 
    \underbrace{\rho_z^+(s_2)-\rho_z^+(t^*)}_{=\varepsilon(s_2-t^*)} 
    + \underbrace{\rho_z^+(t^*) 
    - \rho_z^+(s)}_{\le \varepsilon(t^*-s)} \le \varepsilon(s_2-s),
  \end{equation*}
  contradicting the definition of $t^*$, thus finishing the proof of
  the claim. Consider the statement for $\Bp_z$. Let
  \begin{equation*}
    \mathscr{C}^+ := \seq{t \in (0,T]\,:\, \partial_- \rho_z^+(t) > 0} \cup \seq{0}.
  \end{equation*}
  Let $t \in \Bp_z \setminus \seq{0}$. Suppose there exits $\delta >
  0$ such that for all $0 < t-s < \delta$, $\partial_-\rho_z^+(s) =
  0$. Then, as $\rho_z^+$ is constant on $(t-\delta,t)$, this
  contradicts the fact that $t \in \Bp_z \setminus
  \seq{0}$. Consequently, there exists a sequence $t_n \uparrow t$ as
  $n \rightarrow \infty$ with $t_n \in \mathscr{C}^+$ for all $n \ge
  1$. Hence $\Bp_z \subset \mathrm{cl}_-(\mathscr{C}^+)$. Next,
  suppose that $t \in \mathscr{C}^+ \setminus \seq{0}$. As
  $\partial_-\rho_z^+(t) > 0$ it follows that $\rho_z^+(s) <
  \rho_z^+(t)$ for all $s < t$. Hence,
  \begin{equation*}
    \inf\seq{s \in [0,T]\,:\, \rho_z^+(s) \ge \rho_z^+(t)} = t,
  \end{equation*}
  so that $\mathscr{C}^+ \subset \Bp_z$ by (i). We then conclude, by
  (ii), that $\mathrm{cl}_-(\mathscr{C}^+) \subset \Bp_z$. The proof
  for $\Bm_z$ is analogous.
\end{proof}

\begin{lemma}\label{lemma:OSBC1Path}
Suppose \eqref{eq:Af} holds and $u_0 \in (L^\infty \cap L^1 \cap BV)(\R)$ 
satisfies \eqref{eq:LowUpOnInitial}. Fix a piecewise linear path $z(t)$with $z(0) = 0$. 
Let $u$ be the corresponding entropy solution to \eqref{eq:sscl}, and 
$\rho^+_z,\rho^-_z$ be defined by \eqref{eq:TruncRunningMaxMin}. 
Then, for any $t \in [0,T]$,
 \begin{equation*}
     -\frac{1}{\rho_z^+(t)-z(t)} \le 
     \frac{f'(u(t,y\pm))-f'(u(t,x\pm))}{y-x} \le \frac{1}{z(t)-\rho_z^-(t)},
     \quad x,y\in \R, \, x<y.
  \end{equation*}
  We apply the convention $(0)^{-1} = \infty$.
\end{lemma}

\begin{proof}
  Let $0 \le t \le T$ be given. By assumption there exists a finite
  sequence $\seq{t_n}_{n=0}^N$, $0 = t_0 <\dots < t_N = T$, such that the graph 
  of $z$ is a straight line on each interval
  $[t_n,t_{n+1}]$, $n= 0,\dots, N-1$. Without loss of generality, we 
  will prove the result for $t\in \seq{t_n}_{n=0}^N$. 
  Let $P_n$ be the statement of the
  lemma for $t = t_n$. We need to prove that $P_{n+1}$ holds
  given the validity of $P_n$, where $P_n$ precisely reads
  \begin{equation*}
    -\underbrace{\frac{1}{\rho^+_z(t_n)-z(t_{n})}}_{M_-^n} \le
    \underbrace{\frac{f'(u(t_{n},y\pm))-f'(u(t_{n},x\pm))}{y-x}}_{X_n}
    \le  \underbrace{\frac{1}{z(t_{n})-\rho_z^-(t_n)}}_{M_+^n}, \quad x<y.
  \end{equation*}
 By Lemma~\ref{lemma:OneSidedBounds},
  \begin{equation*}
    -\Theta\left(\frac{1}{(M_-^n)^{-1}-\Delta z_n}\right) \le X_{n+1}
    \le \Theta\left(\frac{1}{(M_+^n)^{-1} + \Delta z_n}\right), 
  \end{equation*}
  where $\Delta z_n =z(t_{n+1})-z(t_n)$. Therefore,
  \begin{equation*}
    -\Theta\left(\frac{1}{\rho^+_z(t_n)-z(t_{n+1})}\right) 
    \le X_{n+1} \le \Theta\left(\frac{1}{z(t_{n+1})-\rho^-_z(t_n)}\right).
  \end{equation*}
 We consider three different cases:
  \begin{displaymath}
   \text{(i)}: t_{n+1} \in \Bp_z, \quad \text{(ii)}: t_{n+1} \in \Bm_z, \quad 
   \text{(iii)}:t_{n+1} \notin \Bp_z \cup \Bm_z.
  \end{displaymath}
 Consider (i). By Lemma~\ref{lemma:BpBmProp}~(i), $z(t_{n+1}) = \rho^+_z(t_{n+1})$. 
 If $(t_n,t_{n+1}] \cap \Bm_z \neq \emptyset$, then $t_{n+1} \in \Bm_z$ by 
 Lemma~\ref{lemma:BpBmProp}~(iv), but this cannot be the case 
 as $\Bp_z \cap \Bm_z = \seq{0}$. Consequently, by 
 Lemma~\ref{lemma:BpBmProp}~(iii), $\rho^-_z(t_n) = \rho^-_z(t_{n+1})$. 
 Hence $P_{n+1}$ follows in case~(i). Case~(ii) is analogous. In case~(iii), we 
 argue as in case~(i) to conclude that $\rho_z^\pm(t_n) = \rho_z^\pm(t_{n+1})$. 
 This proves $P_{n+1}$ in case~(iii). It remains to observe that $P_0$ 
 holds by assumption \eqref{eq:LowUpOnInitial}.
\end{proof}

\subsection{Equivalence}
Let us first, for convenience, collect some consequences of the 
above results in terms of the equivalence relation, see Definition~\ref{def:EquivOfPaths}.
\begin{corollary}\label{cor:PiecewiseEquiv}
  Assume that $u_0 \in (L^\infty \cap L^1 \cap BV)(\R)$. Let $z_1$ and $z_2$
  be Lipschitz continuous paths on $[t_1,t_2]$, and set 
  $\Delta z_i =z_i(t_2)-z_i(t_1)$, $i = 1,2$. Then the following three statements hold:
  \begin{enumerate}
  \item If $\Delta z_1 = \Delta z_2$ and both paths are monotone. Then
    $z_1 \sim z_2$ on  $[t_1,t_2]$ with initial condition $u_0$. 
  \item If $z_1 = z_2 \circ \alpha$ for some nondecreasing,
    Lipschitz continuous, surjective map $\alpha:[t_1,t_2] \rightarrow
    [t_1,t_2]$, then $z_1 \sim z_2$ on $[t_1,t_2]$ with initial
    condition $u_0$.
  \item If $u_0$ satisfies \eqref{eq:LowUpOnInitial} and 
    \begin{equation*}
      -(M_+)^{-1} \le z_i(t)-z_i(t_1) \le (M_-)^{-1}, \quad \text{for
        all} \quad t_1 \le t \le t_2, 
    \end{equation*}
    and $\Delta z_1 = \Delta z_2$, then $z_1 \sim z_2$ on $[t_1,t_2]$
    with initial condition $u_0$.
  \end{enumerate}
\end{corollary}
\begin{proof}
  Let $u^1,u^2$ denote the entropy solutions to \eqref{eq:sscl}
  associated with the paths $z_1$, $z_2$ and initial condition $u_0$.
  
  \noindent\emph{$(1)$}. Suppose $z_1$ and $z_2$ are nondecreasing. We
  need to show that $u^1(t_2) = u^2(t_2)$. 
  To apply Lemma~\ref{lemma:EntSub}, assume
  for notational ease $[t_1,t_2] = [0,\tau]$. Let $\tilde{z}(t) = t$
  in Lemma~\ref{lemma:EntSub}. Then $u^i(t,x) = v(z_i(t),x)$, $i = 1,2$,
  where $v$ is the entropy solution to
  \begin{equation*}
    \partial_tv + \partial_xf(v) = 0, \qquad v(z(0),x) = u_0(x).
  \end{equation*}
  It follows that $u^1(\tau) = v(z_1(\tau)) = v(z_2(\tau)) =
  u^2(\tau)$. If $z_1,z_2$ are nonincreasing, we consider
  instead $\tilde{z}(t) = -t$ and the paths $-z_1,-z_2$, and proceed as above. 

  \noindent\emph{$(2)$}. By Lemma~\ref{lemma:EntSub}, 
  $u(t,x) := u^2(\alpha(t),x)$ satisfies
  \begin{equation*}
    \partial_tu + (z_2 \circ \alpha)'\partial_xf(u) = 0, \qquad u(0,x)
    = u^2(\alpha(t_1),x) = u_0(x). 
  \end{equation*}
  But then $u = u^1$, and so $u^1(t_2) = u^2(\alpha(t_2)) =
  u^2(t_2)$.

  \noindent\emph{$(3)$}. Set $\tilde{z}_i(t) = z_i(t+t_1)-z_i(t_1)$, and apply 
  Corollary~\ref{cor:RegInitPathSub} with $\tilde{z}_i$, $i = 1,2$. This yields
  \begin{equation*}
    u^1(t_2) = \tilde{u}^1(t_2-t_1) = v(\tilde{z}_1(t_2-t_1)) =
    v(\tilde{z}_2(t_2-t_1)) = \tilde{u}^2(t_2-t_1) = u^2(t_2),  
  \end{equation*}
  where $\tilde{u}^1$ and $\tilde{u}^2$ are the entropy solutions to
  \eqref{eq:sscl} with $z = \tilde{z}_1$ and $z = \tilde{z}_2$, respectively.
\end{proof}

Next we provide preliminary version of Theorem \ref{thm:GeneralEquiv}. 
 
\begin{lemma}\label{lemma:EquivalencePiecewise}
  Suppose $f$ satisfies \eqref{eq:Af}. Let $z_1$, $z_2$ be 
  piecewise linear, continuous paths on
  $[0,\tau]$, satisfying $z_1(0) = z_2(0) = 0$ and 
  $z_1(\tau) = z_2(\tau)$. Suppose there exist 
  piecewise linear nondecreasing surjective maps 
  $\alpha_i:[0,\tau] \rightarrow [0,\tau]$, $i =1,2$,
  such that $\rho^\pm_{z_1} \circ \alpha_1 = \rho^\pm_{z_2} \circ
  \alpha_2$. Let $u^1$ and $u^2$ be entropy solutions to
  \eqref{eq:sscl} with initial condition $u_0 \in (L^\infty \cap L^1
  \cap BV)(\R)$ satisfying \eqref{eq:LowUpOnInitial}. 
  Then $u^1(\tau) = u^2(\tau)$. That is, $z_1 \sim
  z_2$ on $[0,\tau]$ for all $u_0 \in (L^\infty \cap L^1 \cap BV)(\R)$
  satisfying \eqref{eq:LowUpOnInitial}.
\end{lemma}

\begin{proof}
  Set $\tilde{z}_i = z_i \circ \alpha_i$, $i = 1,2$. By
  Corollary~\ref{cor:PiecewiseEquiv}~(2), $\tilde{z}_i \sim z_i$.
  Furthermore, $\tilde{z}_i$ is piecewise linear with
  $\tilde{z}_i(0) = 0$, and $\tilde{z}_1(\tau) = \tilde{z}_2(\tau)$. 
  Hence, we might as well assume $\rho^\pm_{z_1} =
  \rho^\pm_{z_2} = \rho^\pm$ which also implies $\Bpm_{z_1} =
  \Bpm_{z_2} =: \Bpm$.

  By Lemma~\ref{lemma:BpBmProp}~(iv),
  \begin{equation*}
    \Bm \cup \Bp \setminus \seq{0}= \bigcup_{i = 1}^N (s_i,t_i],
    \quad t_0 = 0 \le s_1 < t_1 < \dots < s_N < t_N \le \tau =s_{N+1},
  \end{equation*}
  for some $0 \le N < \infty$. Suppose $u^1(t_n) = u^2(t_n)$. We want
  to show that $u^1(s_{n+1}) = u^2(s_{n+1})$. Assume $t_n < \tau$,
  for otherwise $t_n = \tau = s_{n+1}$ and we are done. Then $(t_n,s_{n+1}]
  \cap (\Bm \cup \Bp) = \emptyset$. By
  Lemma~\ref{lemma:BpBmProp}~(iii),
  \begin{equation}\label{eq:PathBoundSmooth}
    \rho^-(t_n) \le z_1(t),z_2(t) \le \rho^+(t_n), 
  \end{equation}
  for all $t \in [t_n,s_{n+1}]$. Moreover, due to
  Lemma~\ref{lemma:BpBmProp}~(i), $z_1(t_n) = z_2(t_n) =: a$ 
  and $z_1(s_{n+1}) = z_2(s_{n+1})$. 
  By Lemma~\ref{lemma:OSBC1Path},
  \begin{equation*}
    -\underbrace{\frac{1}{\rho^+(t_n)-a}}_{M_-} 
    \le \frac{f'(u(t_n,y\pm))-f'(u(t_n,x\pm))}{y-x} 
    \le \underbrace{\frac{1}{a-\rho^-(t_n)}}_{M_+}, \quad x<y.
  \end{equation*}
  By Corollary~\ref{cor:PiecewiseEquiv}~(3), since $z_1$ and $z_2$
  satisfy \eqref{eq:PathBoundSmooth}, it follows that $z_1 \sim z_2$ on
  $[t_n,s_{n+1}]$, and so $u^1(s_{n+1}) = u^2(s_{n+1})$.

  Suppose $u^1(s_n) = u^2(s_n)$. Then, by the continuity of $z$,
  $(s_n,t_n] \subset \Bp$ or $(s_n,t_n] \subset \Bm$. By
  Lemma~\ref{lemma:BpBmProp}~(i) $z_1(t) = z_2(t)$ on $[s_n,t_n]$, and
  so $u^1(t_n) = u^2(t_n)$.
\end{proof}

In order to apply the above lemma we need to know when 
there exist suitable $\alpha_1$ and $\alpha_2$. 
This is answered by the following observation.
\begin{lemma}\label{lemma:PiecewisePrecomEquiv}
  Let $\rho_1,\rho_2$ be nondecreasing and continuous on $[0,\tau]$
  satisfying $\rho_1(0) = \rho_2(0)$ and $\rho_1(\tau) =
  \rho_2(\tau)$. Let
 \begin{displaymath}
  \mathscr{S} := \seq{(s_1,s_2)\in [0,\tau]^2\,:\, \rho_1(s_1) = \rho_2(s_2)}.
 \end{displaymath}
Suppose there exist partitions $0 = s_i^0 < s_i^1 < \cdots < s_i^{N_i} = \tau$ 
such that for each rectangle $Q_{j,k} = [s_1^j,s_1^{j+1}] \times [s_2^k,s_2^{k+1}]$ 
satisfying $Q_{j,k} \cap \mathscr{S} \neq \emptyset$,  
  \begin{itemize}
   \item[(i)]  $[s_1^j, s_1^{j+1}] \subset \mathrm{cl}_+(\Bp_{\rho_1})$ 
   or $[s_2^k, s_2^{k+1}] \subset \mathrm{cl}_+(\Bp_{\rho_2})$ where $\mathrm{cl}_+$ 
   denotes the closure with respect to decreasing sequences, or
   \item[(ii)] $(s_1^j, s_1^{j+1}] \cap \Bp_{\rho_1} = \emptyset$ or $(s_2^k, s_2^{k+1}] \cap \Bp_{\rho_2} 
   = \emptyset$.
  \end{itemize}
  Then there exists $\alpha_1,\alpha_2 \in \mathscr{A}_\tau$
  such that $\rho_1 \circ \alpha_1 = \rho_2 \circ \alpha_2$.
\end{lemma}
\begin{proof}
Suppose we can find a continuous path $\alpha = (\alpha_1,\alpha_2)$ with 
$\alpha_1,\alpha_2 \in \mathscr{A}_\tau$ satisfying $\alpha(t) \in \mathscr{S}$ 
for all $t \in [0,\tau]$. Then $\rho_1 \circ \alpha_1(t) 
= \rho_2 \circ \alpha_2(t)$, and we are done.
Define the "lower" and "upper" boundary of $Q_{j,k}$ by 
\begin{align*}
	\partial^-Q_{j,k} &= \seq{s_1^j} \times [s_2^k,s_2^{k+1}] \cup [s_1^j,s_1^{j+1}] 
	\times \seq{s_2^k}, \\
	\partial^+Q_{j,k} &= \seq{s_1^{j+1}} \times [s_2^k,s_2^{k+1}] \cup [s_1^j,s_1^{j+1}] 
	\times \seq{s_2^{k+1}}.
\end{align*} 
\emph{Claim 1: } Let $\tilde s \in \partial^-Q_{j,k} \cap \mathscr{S}$. 
Then there exist a (continuous) path $\gamma:[0,1] 
\rightarrow \mathscr{S} \cap Q_{j,k}$ such that 
$\gamma(0) = \tilde s$ and $\gamma(1) \in \partial^+Q_{j,k} \cap \mathscr{S}$.
 
 Before proving the claim, let us see why the result follows. Let $\mathcal{D}$ 
 be any finite union of squares, i.e., $\mathcal{D} = \cup_{(j,k) \in \mathcal{I}}Q_{j,k}$, 
 where $\mathcal{I} \subset \seq{(j,k)\,:\, 0 \leq j \leq N_1, 0 \leq k\leq N_2}$. 
 Then, for such $\mathcal{D}$, let 
 the lower and upper boundary be defined by 
 \begin{displaymath}
 	\partial^\pm\mathcal{D} = \bigcup_{(j,k) \in \mathcal{I}} 
	\partial^\pm Q_{j,k} \setminus \bigcup_{(j,k) \in \mathcal{I}}\partial^\mp Q.
 \end{displaymath}
Proceeding by induction on the number of boxes we extend, upon 
concatenating paths, the result in Claim~1 to any such domain $\mathcal{D}$. 
But, then Claim~1 is valid for $\mathcal{D} = [0,\tau] \times [0,\tau]$.  
Consequently, as $(0,0) \in \partial^-\mathcal{D} \cap \mathscr{S}$ 
there exists a continuous path $\gamma$ satisfying $\gamma(0) = (0,0)$ and 
$\gamma(1) \in \partial^+\mathcal{D} \cap \mathscr{S} = ([0,\tau] 
\times \seq{\tau} \cup \seq{\tau} \times [0,\tau] ) \cap \mathscr{S}$. 
To finish the proof we apply the following observation with $\gamma(1) = s$ 
and $(\tau,\tau) = \tilde s$. \\
 \emph{Claim 2: }
 Suppose $s,\tilde s \in \mathscr{S}$. Then the point
  \begin{equation}\label{eq:MaxInS}
	s^* := (\max\seq{s_1,\tilde s_1},\max\seq{s_2,\tilde s_2}) \in \mathscr{S}.
  \end{equation}
 Furthermore, the straight line connecting $s$ and $s^*$ belongs to $\mathscr{S}$. 
 Similarly for $\tilde s$ and $s^*$. 
 
 It remains to prove Claim~1 and 2. \\
\emph{Proof of Claim~1.}
 Let $g_i(\xi) := \min\seq{r \geq \tilde s_i\,:\rho_i(r) \geq \xi}$ for $i = 1,2$. 
 Consider the two paths
 \begin{displaymath}
  \gamma_1(\lambda) = \left(\lambda,g_2(\rho_1(\lambda))\right), 
  \quad \lambda \in [\tilde s_1,s_1^{j+1}],
 \mbox{ and } 
  \gamma_2(\lambda) = \left(g_1(\rho_2(\lambda)),\lambda \right), 
  \quad \lambda \in [\tilde s_2,s_2^{k+1}].
 \end{displaymath}
 Note $\rho_i(g_i(\xi)) = \max\seq{\xi,\rho_i(\tilde s_i)}$, and so, 
 $\gamma_1:[\tilde s_1,s_1^{j+1}] \rightarrow \mathscr{S}$ and 
 $\gamma_2:[\tilde s_2,s_2^{k+1}] \rightarrow \mathscr{S}$. 
 Furthermore, $\gamma_1(\tilde s_1) 
 = \gamma_2(\tilde s_2) = \tilde s$. Suppose $(s_1^j,s_1^{j+1}] 
 \cap \Bp_{\rho_1} = \emptyset$. Then, by Lemma~\ref{lemma:BpBmProp}~(iii) $\rho_1$ 
 is constant on $[s_1^j,s_1^{j+1}]$. Hence, $\gamma_1(\lambda) = (\lambda,\tilde s_2)$, a 
 straight line connecting $\tilde s$ and $\partial^+Q_{j,k}$. 
 Similarly, if $(s_2^k,s_2^{k+1}] \cap \Bp_{\rho_2} = \emptyset$. To treat case (i), we 
 need the following observation.
 \begin{itemize}
  \item [(1)]$g_i$ is left continuous.
  \item [(2)]$g_i$ is right continuous at $\xi$, if for any $\bar \kappa > g_i(\xi)$, 
  $(g_i(\xi),\bar{\kappa}] \cap \Bp_{\rho_i} \neq \emptyset$.
 \end{itemize}
Suppose $[s_1^j, s_1^{j+1}] \subset \mathrm{cl}_+(\Bp_{\rho_1})$. 
It thus follow that $g_1$ is continuous as long as $g_1(\rho_2(\lambda)) \in [s_1^j,s_1^{j+1}]$. 
Consequently, $\gamma_2$ is a continuous path 
connecting $\tilde s$ and $\partial^+Q_{j,k}$. 
Similarly, if $[s_2^k, s_2^{k+1}] \subset \mathrm{cl}_+(\Bp_{\rho_2})$.

Let us prove (1) and (2). Let $\xi,\seq{\xi_n}_{n \geq 0} \subset [\rho_i(\tilde s_i),\rho_i(\tau)]$ 
and suppose $\xi_n \uparrow \xi$. Let $g_i(\xi) = \kappa$ and $\kappa_n = g_i(\xi_n)$. 
We want to show that $\kappa_n \uparrow \kappa$. Clearly, $\kappa_n \leq \kappa$ 
for all $n \geq 1$. For any $\varepsilon > 0$, there 
exists $n_0(\varepsilon)$ s.t. $\xi - \varepsilon \leq \xi_n$ 
for all $n \geq n_0(\varepsilon)$. Consequently, for all $n \geq n_0(\varepsilon)$,
 \begin{displaymath}
  \kappa_n = \min\seq{r \geq \tilde s_i\,:\rho_i(r) \geq \xi_n} 
	  \geq \min\underbrace{\seq{r \geq \tilde s_i\,:\rho_i(r) 
	  + \varepsilon \geq \rho_i(\kappa)}}_{V^\varepsilon}.
 \end{displaymath}
 Hence, if $\seq{\kappa_n}_{n \geq 0}$ has an accumulation point $\tilde \kappa < \kappa$, 
 then $\tilde \kappa \in V^\varepsilon$ for all $\varepsilon > 0$, i.e., 
 $\rho_i(\tilde \kappa) \geq \rho_i(\kappa)$, contradicting that $\kappa = g_i(\xi)$. 
 This finishes the proof of (1). Suppose $\xi_n \downarrow \xi$. 
 We want to show that $\kappa_n \downarrow \kappa$. Again, $\kappa_n \geq \kappa$. 
 For any $\varepsilon > 0$, there exists $n_0(\varepsilon)$ s.t. $\xi + \varepsilon \geq \xi_n$ 
 for all $n \geq n_0(\varepsilon)$. Hence
 \begin{displaymath}
  \kappa_n = \min\seq{r \geq \tilde s_i\,:\rho_i(r) \geq \xi_n} 
	  \leq \min\seq{r \geq \tilde s_i\,:\rho_i(r) \geq \rho_i(\kappa) + \varepsilon}.
 \end{displaymath}
 Suppose $\seq{\kappa_n}_{n \geq 0}$ has an accumulation 
 point $\bar \kappa > \kappa$. Then
 \begin{displaymath}
  k_0 := \inf\seq{r \geq \tilde s_i\,:\rho_i(r) > \rho_i(\kappa)} \geq \bar \kappa.
 \end{displaymath}
 But then $\rho_i(\kappa_0) = \rho_i(\kappa)$ and so $(\kappa,\kappa_0] 
 \cap \Bp_{\rho_i} = \emptyset$, a contradiction. This finishes the proof of (2).
 
\emph{Proof of Claim~2.} To this end, we might as    well assume $s_1 \leq \tilde s_1$ 
and $s_2 \geq \tilde s_2$, the other cases beeing either trivial, or 
analogous. As $\rho_2$ is nondecreasing
  \begin{displaymath}
	\rho_1(\tilde s_1) = \rho_2(\tilde s_2) \leq \rho_2(s_2) = \rho_1(s_1),
  \end{displaymath}
  but then, as $\rho_1$ is nondecreasing, $\rho_1(\tilde s_1) = \rho_1(s_1)$. 
  Similarly, $\rho_2(\tilde s_2) = \rho_2(s_2)$. Consequently \eqref{eq:MaxInS} follows. 
  Likewise, it is easily seen that the straight line 
  connecting $s$ and $s^*$ belongs to $\mathscr{S}$. 
  Similarly for $\tilde s$ and $s^*$.
\end{proof}

\begin{lemma}\label{lemma:FiniteOrm}
  Let $z \in C_0([0,\tau])$ and suppose $0 \le \min\seq{M_-,M_+} <
  \infty$. Let $\seq{\tau_n}_{n \ge 0}$ be defined in Definition~\ref{def:Orm}. 
  Then $\tau_N = 0$ for some $1 \le N <\infty$.
\end{lemma}

\begin{proof}
  Note that $0 \le \tau_n \le \tau_{n-1} \le \tau$ for all $n \ge
  1$. And so, $\tau_n \downarrow \tau^*$ for some $\tau^* \in
  [0,\tau]$. Suppose $\tau_n > 0$ for all $n \ge 0$. Let $z_n :=
  z(\tau_n)$, $n \ge 0$. Consider the case $\tau_1 \in \Bm_z$. By
  Lemma~\ref{lemma:BpBmProp}~(i), $z_1 = \rho^-_z(\tau_1)$. Then $z_2 =
  \rho^+_z(\tau_2), z_3 = \rho^-_z(\tau_3), \dots$, and generally
  \begin{equation*}
    z_{2n} = \rho^+_z(\tau_{2n}) \quad \text{and} 
    \quad z_{2n-1} = \rho^-_z(\tau_{2n-1}), \qquad n\ge 1.
  \end{equation*}
  Consequently, $z_{2n} \ge (M_-)^{-1}$ and
  $z_{2n-1} \le -(M_+)^{-1}$ implying that $z(\tau_n) \nrightarrow
  z(\tau^*)$ as $n \rightarrow \infty$, contradicting the continuity of $z$. 
  The case $\tau_1 \in \Bp_z$ is analogous. 
  As a result, there exists $0 \le N < \infty$ suchthat $\tau_N = 0$.
\end{proof}

\begin{lemma}\label{lemma:InitApprox}
  Suppose \eqref{eq:Af} holds and $u_0 \in (L^1 \cap
  L^\infty)(\R)$ satisfies \eqref{eq:LowUpOnInitial}. Then there exist a family
  $\seq{\ue_0}_{\varepsilon > 0} \subset (L^1 \cap L^\infty \cap
  BV)(\R)$ such that $\ue_0 \rightarrow u_0$ in $L^1(\R)$ as
  $\varepsilon \downarrow 0$ and for $y<x$,
  \begin{equation}\label{eq:LipMollBound}
    -\underbrace{\min
      \seq{M_-,\frac{c}{\varepsilon}\norm{h(u_0)}_{\infty}}}_{M_-^\varepsilon}
    \le  \frac{f'(u_0^\varepsilon(y))-f'(u_0^\varepsilon(x))}{y-x} \le
    \underbrace{\min
      \seq{M_+,\frac{c}{\varepsilon}\norm{h(u_0)}_{\infty}}}_{M_+^\varepsilon}, 
  \end{equation}
  where $h(u) = \int_0^u f''(z)\,dz$ and $c > 0$ is a constant independent 
  of $f,u_0, \varepsilon$.
\end{lemma}
\begin{proof}
  Let $J_\varepsilon(x) = \varepsilon^{-1}J(\varepsilon^{-1}x)$, where $J
  \in C^\infty_c(\R)$ satsifies $\mathrm{supp}(J) \subset (-1,1), \, J
  \ge 0, \mbox{ and } \int J = 1$. Let $g$ be the inverse of $h$,
  i.e., $h(g(u)) = u$. Define
  \begin{equation*}
    u_0^\varepsilon(x) := g((h(u_0) \star J_\varepsilon)(x)).
  \end{equation*}
  As $u_0 \in L^\infty(\R)$, it follows by \eqref{eq:Af} that there 
  exist $0 < \alpha \le \beta < \infty$ such that $\alpha \le h'(u) \le \beta$ for 
  all $\abs{u} \leq \norm{u_0}_\infty$. Consequently, 
   $\beta^{-1} \le g'(v) \le \alpha^{-1}$, for all $v \in \seq{h(u)\,:\,\abs{u} \leq  \norm{u_0}_\infty}$. 
   Also, $g(0) = 0$. By Young's inequality for convolutions,
  \begin{equation*}
    \norm{\ue_0}_{L^p(\R)} \le \frac{1}{\alpha}\norm{h(u_0)
    \star J_\varepsilon}_{L^p(\R)}
    \le \frac{1}{\alpha}\norm{h(u_0)}_{L^p(\R)}
    \norm{J_\varepsilon}_{L^1(\R)}
    \le \frac{\beta}{\alpha}\norm{u_0}_{L^p(\R)},
  \end{equation*}
  for any $p\in [1, \infty]$. Similarly,
  \begin{align*}
    \norm{u_0^\varepsilon}_{BV} &\le \frac{1}{\alpha}\norm{h(u_0) 
    \star J_\varepsilon}_{BV} \\
    &\le \frac{1}{\alpha \varepsilon} \norm{h(u_0)}_{L^1(\R)}\norm{J'}_{L^1(\R)} \\
    &\le \frac{\beta}{\alpha \varepsilon}
    \norm{u_0}_{L^1(\R)}\norm{J'}_{L^1(\R)}.
  \end{align*}
  To prove \eqref{eq:LipMollBound} we
  observe that
  \begin{align*}
    \abs{f'(u_0^\varepsilon(y))-f'(u_0^\varepsilon(x))}
    &=  \abs{h(u_0^\varepsilon(y))-h(u_0^\varepsilon(x))} \\
    &= \abs{(h(u_0) \star J_\varepsilon)(y)-(h(u_0) \star J_\varepsilon)(x)} \\
    &= \abs{\int_\R h(u_0(\xi))\left(J_\varepsilon(y-\xi)-J_\varepsilon(x-\xi)\right)\,d\xi} \\
    &\le
    \frac{1}{\varepsilon}\norm{h(u_0)}_{L^\infty(\R)}\norm{J'}_{L^1(\R)}\abs{y-x}.
  \end{align*}
  On the other hand,
  \begin{equation*}
    \frac{f'(u_0^\varepsilon(y))-f'(u_0^\varepsilon(x))}{y-x} 
    = \int_\R \frac{f'(u_0(y-\xi))-f'(u_0(x-\xi))}{(y-\xi)-(x-\xi)}J_\varepsilon(\xi)\,d\xi.
  \end{equation*}
  Combining the two yields \eqref{eq:LipMollBound}. Finally, for any
  $1 \le p \le \infty$,
  \begin{equation*}
    \norm{u_0^\varepsilon - u_0}_{L^p(\R)} 
    \le \frac{1}{\alpha}\norm{h(u_0) \star J_\varepsilon - h(u_0)}_{L^p(\R)},
  \end{equation*}
  and so $u_0^\varepsilon \rightarrow u_0$ in $L^p(\R)$.
\end{proof}

\begin{proof}[Proof of Theorem \ref{thm:OrmEquiv}]
  Assume $u_0 \in BV(\R)$. We construct an approximation of the path $z$
  by supplementing the interpolation points
  $\seq{(\tau_m,z(\tau_m))}_{m = 0}^{N}$, given in Definition~\ref{def:Orm}, by points
  $\seq{(t_m,z(t_m))}_{m = 1}^\infty$, This results in a sequence of
  piecewise linear approximations to $z$, with $z_0 =
  \orm_{\tau,M_\pm}(z)$ and $z_n$ the piecewise linear
  interpolation based on the points $\seq{(\tau_m,z(\tau_m))}_{m = 0}^{N} \cup
  \seq{(t_m,z(t_m))}_{m = 1}^n$. We claim
  \begin{equation}\label{eq:SimApproxClaim}
  \begin{split}
    &\text{For all $n \ge 1$, $z_n \sim z_0$ on $[0,\tau]$ for all}
    \\ & \text{$u_0 \in (L^1 \cap L^\infty \cap BV)(\R)$ 
    satisfying \eqref{eq:LowUpOnInitial}}.
  \end{split}
  \end{equation}
  Before proving the claim, we verify that Theorem \ref{thm:OrmEquiv} follows. 
  As $z$ is uniformly continuous on $[0,\tau]$ we may pick the sequence
  $\seq{t_n}_{n = 0}^\infty$ such that $z_n \rightarrow z$ uniformly
  on $[0,\tau]$. Let $u^n$ be the solution to \eqref{eq:sscl} with
  path $z^n$ and initial condition $u_0$.  By \eqref{eq:stabilityPes}
  it follows that $u^n(\tau) \rightarrow u(\tau)$ in $L^1(\R)$. But by
  \eqref{eq:SimApproxClaim}, $u^n(\tau) = u^0(\tau)$ for all $n \ge
  0$. Hence $u^0(\tau) = u(\tau)$ which proves that the equivalence holds
  for all initial functions $u^0 \in (L^1 \cap L^\infty \cap BV)(\R)$.

  To see that it suffices with $u_0 \in (L^\infty \cap L^1)(\R)$ let
  $\seq{\ue_0}_{\varepsilon > 0}$ be the approximation of $u_0$
  obtained in Lemma~\ref{lemma:InitApprox}. Let $\ue$ denote the
  solution to \eqref{eq:sscl} with path $z$ and initial condition
  $\ue_0$, and $u^{\varepsilon,0}$ be the solution to \eqref{eq:sscl}
  with path $\orm_{\tau,M_\pm}(z)$ and initial condition $\ue_0$. Due
  to the above and \eqref{eq:LipMollBound}, 
  $\ue(\tau) = u^{\varepsilon,0}(\tau)$ for all $\varepsilon > 0$. By the
  continuous dependence estimate \eqref{eq:stabilityPes} and the
  triangle inequality,
  \begin{align*}
    \norm{u(\tau)-u^0(\tau)}_{L^1(\R)}
    &\le \norm{u(\tau)-\ue(\tau)}_{L^1(\R)} 
    + \norm{u^{\varepsilon,0}(\tau)-u^0(\tau)}_{L^1(\R)} \\
    &\le 2\norm{u_0-\ue_0}_{L^1(\R)},
  \end{align*}
  from which Theorem \ref{thm:OrmEquiv} follows. 
  
  It remains to prove the claim \eqref{eq:SimApproxClaim}.
  %\emph{Proof of claim.}  
  Suppose $z_{n-1} \sim z_0$, we must show that
  $z_n \sim z_0$. For some $1 \le m \le N$, we have $\tau_m \le t_n <
  \tau_{m-1}$. Note that $ z_n(\tau_{m}) = z_{n-1}(\tau_m) =
  z(\tau_m)$ and $z_n(\tau_{m-1}) = z_{n-1}(\tau_{m-1}) =
  z(\tau_{m-1})$, so $z_n(t) = z_{n-1}(t)$ for all $t \in
  [0,\tau_{m}]\cup[\tau_{m-1},\tau]$.

  Consider the case $\tau_{m} \in \Bm_z$. By
  Lemma~\ref{lemma:PiecewisePrecomEquiv} there exist piecewise linear
  monotone surjective functions $\alpha_i:[\tau_m,\tau_{m-1}]
  \rightarrow [\tau_m,\tau_{m-1}]$, $i = 1,2$, such 
  that $\rho_{z_n}^+ \circ \alpha_1 = \rho_{z_{n-1}}^+ \circ \alpha_2$ on
  $[\tau_m,\tau_{m-1}]$. We need to show that $\rho_{z_n}^- \circ \alpha_1 = \rho_{z_{n-1}}^- \circ
  \alpha_2$ on $[\tau_m,\tau_{m-1}]$. Recall that $\tau_m = \max\seq{\Bm_z \cap
  [0,\tau_{m-1}]}$, so that $(\tau_m,\tau_{m-1}] \cap \Bm_z = \emptyset$. By  	
  Lemma~\ref{lemma:BpBmProp}~(iii), $\rho_z^-$ is constant on $[\tau_m,\tau_{m-1}]$. 
  Hence, $z(t) \ge \rho_z^-(t) = \rho_z^-(\tau_m) =
  z(\tau_m)$ for $t \in (\tau_m,\tau_{m-1}]$. Accordingly, for any $t \in
  (\tau_m,\tau_{m-1}]$ and all $n \ge 0$, $z_n(t) \ge z(\tau_m) = z_n(\tau_m)$, so that
  \begin{displaymath}
   \rho_{z_n}^-(t) = \min \seq{-\frac{1}{M_+},\min_{0 \leq s \leq t}\seq{z_n(s)}} 
   = z(\tau_m), \qquad t \in [\tau_m,\tau_{m-1}].
  \end{displaymath}
  It follows that $\rho_{z_n}^- \circ \alpha_1 = \rho_{z_{n-1}}^- \circ
  \alpha_2$ on $[\tau_m,\tau_{m-1}]$. Hence, upon taking
  \begin{equation*}
    \tilde{\alpha}_i(t) 
    =
    \begin{cases}
      \alpha_i(t) &\text{for $t \in [\tau_m,\tau_{m-1}]$},\\
      t &\text{otherwise,}
    \end{cases}
  \end{equation*}
  we have $\rho_{z_n}^\pm \circ \tilde{\alpha}_1 = \rho_{z_{n-1}}^\pm \circ \tilde{\alpha}_2$ 
  on $[0,\tau]$. By Lemma~\ref{lemma:EquivalencePiecewise}, it follows that $z_n \sim
  z_{n-1} \sim z_0$. The case $\tau_{m} \in \Bp_z$ is treated similarly. 
\end{proof}

\begin{proof}[Proof of Theorem \ref{thm:OSB}]
  When $\orm_{\tau,u_0}(z)$ is well-defined, the result is an immediate
  consequence of Lemma~\ref{lemma:OSBC1Path} and
  Theorem~\ref{thm:OrmEquiv}. The general case follows by
  approximation. Let $\seq{\ue_0}_{\varepsilon > 0}$ be the
  approximation of $u_0$ obtained in Lemma~\ref{lemma:InitApprox} and
  $\ue$ be the entropy solution to \eqref{eq:sscl} with initial
  condition $\ue_0$. By the continuous dependence 
  estimate \eqref{eq:stabilityPes}, it follows that $\ue(t) \rightarrow
  u(t)$ in $L^1(\R)$ as $\varepsilon \downarrow 0$. Let $0\le \test \in
  C^\infty_c(\R^2)$ satisfy $\mathrm{supp}(\test) \subset
  \seq{(x,y)\,:\, y > x}$ and define
  \begin{equation*}
    \Phi[u,\varphi] := \iint_{\R \times \R}
    \frac{f'(u(\tilde{y}))-f'(u(\tilde{x}))}{\tilde{y}-\tilde{x}}\test(\tilde{x},\tilde{y})
    \,d\tilde{x} d\tilde{y}.
  \end{equation*}
  We have that $\Phi[\ue(t),\varphi] \rightarrow
  \Phi[u(t),\varphi]$ as $\varepsilon\to 0$. By
  Lemmas \ref{lemma:OSBC1Path}, \ref{lemma:FiniteOrm}, and \ref{lemma:InitApprox}, 
  and Theorem~\ref{thm:OrmEquiv}, we conclude that
  \begin{equation*}
    -\frac{1}{\rho^+_z(t)-z(t)}\norm{\test}_{L^1(\R \times \R)} \le
    \Phi[\ue(t),\varphi] \le
    \frac{1}{z(t)-\rho^-_z(t)}\norm{\test}_{L^1(\R \times \R)}, 
  \end{equation*}
  and so the same holds for $u$. Let $J^i_\delta(x) =
  \delta^{-1}J^i(\delta^{-1}x)$, where $J^i \in C^\infty_c(\R)$
  satisfy $J^i \ge 0$  and $\int J^i = 1$, $i =1,2$. 
  We also demand that  $\mathrm{supp}(J^1) \subset (0,1)$ and
  $\mathrm{supp}(J^2) \subset (-1,0)$. Fix $y > x$ and let $\test(\tilde{x},\tilde{y}) =
  J^1_{\delta_1}(x-\tilde{x})J^2_{\delta_2}(y-\tilde{y})$. 
  Sending $\delta_1,\delta_2 \downarrow 0$
  yields the left (essential) limit at $x$ and the right (essential)
  limit at $y$ in Theorem~\ref{thm:OSB}. The general statement follows
  by modifying the definition of $\test$.
\end{proof}

Next, we want to apply Theorem~\ref{thm:OrmEquiv} to 
prove Theorem~\ref{thm:GeneralEquiv}. 
An important step in this direction is the following observation, which is also of 
importance for Theorem~\ref{thm:EqDist}.
\begin{lemma}\label{lemma:OrmCompEquiv}
Let $z \in C_0([0,\tau])$ be a path for which $\orm_{\tau,M_\pm}(z)$ 
is well defined and fix a map $\alpha \in \mathscr{A}_\tau$. 
Denote by $\alpha^{-1}$ the generalized 
inverse of $\alpha$, cf.~\eqref{eq:gen-inv}. 
Then $\mathcal{T}_{z \circ \alpha}^\pm 
= \seq{\alpha^{-1}(t) \,:\, t \in \mathcal{T}_z^\pm}$, cf.~\eqref{eq:TpmDef}.
It follows that $\mathcal{T}_{z \circ \alpha} 
= \seq{\alpha^{-1}(t)\,:t \in \mathcal{T}_z} \cup \seq{\tau}$. 
Furthermore, $\orm_{\tau,M_\pm}(z \circ \alpha) \sim  \orm_{\tau,M_\pm}(z)$ 
for any $u_0$ satisfying \eqref{eq:LowUpOnInitial}.
\end{lemma}

To prove this statement, we need a more technical result.
\begin{lemma}\label{lemma:BpBmComp}
  Let $z \in C_0([0,\tau])$, and fix a nondecreasing 
  surjective map $\alpha:[0,\tau] \rightarrow [0,\tau]$, with 
  generalized inverse $\alpha^{-1}$ \eqref{eq:gen-inv}.
  Then
  \begin{itemize}
  \item[(i)]$\Bpm_{z \circ \alpha} = \alpha^{-1}(\Bpm_z) \cap
    \mathscr{B}_\alpha$, $\mathscr{B}_\alpha := \seq{t \in
      [0,\tau] \,:\, \min\seq{s \,:\,\alpha(s) \ge \alpha(t)} = t}$.
  \item[(ii)] For a given set $S \subset [0,\tau]$ satisfying $\sup
    \seq{S} \in S$
    \begin{equation*}%\label{eq:MaxMinClaim}
      \max \seq{\alpha^{-1}(S) \cap \mathscr{B}_\alpha} = \alpha^{-1}(\max\seq{S}).
    \end{equation*}
  \item[(iii)] If $\alpha(\zeta^*) = \tau^*$ for some $0 \le
    \zeta^*,\tau^* \le \tau$ then
    \begin{equation*}
      [0,\zeta^*] \cap \mathscr{B}_\alpha = \alpha^{-1}([0,\tau^*]) \cap \mathscr{B}_\alpha.
    \end{equation*}
  \end{itemize}
\end{lemma}
\begin{proof}
(i). Consider $\Bp_{z \circ \alpha}$, with the proof for $\Bm_{z \circ \alpha}$ being
  analogous. Since $\alpha$ is nondecreasing,
  \begin{equation*}
    \rho^+_{z \circ \alpha}(t) = \max \seq{\frac{1}{M_-}, 
    \max_{0 \le s \le \alpha(t)}\seq{z(s)}} = \rho^+_z(\alpha(t)).
  \end{equation*}
  Hence,
  \begin{equation}
    \Bp_{z \circ \alpha} = \seq{t \in [0,\tau]\,:\, \inf\seq{s\,:\,
        \rho_z^+(\alpha(s))\ge \rho_z^+(\alpha(t))} = t}. 
  \end{equation}
  Assume $t \in \alpha^{-1}(\Bp_{z}) \cap \mathscr{B}_\alpha$ and 
  $\rho_z^+(\alpha(s))\ge \rho_z^+(\alpha(t))$. Then, as
  $\alpha(t) \in \Bp_z$, $\alpha(s) \ge \alpha(t)$. Since $t \in
  \mathscr{B}_\alpha$ it follows that $t \leq s$. Consequently, 
  $t \in \Bp_{z \circ \alpha}$. Conversely, if 
  $t \notin \mathscr{B}_\alpha$, then there 
  exists $s < t$ such that $\alpha(s) = \alpha(t)$, 
  implying that $t \notin \Bp_{z \circ
  \alpha}$. Similarly, if $\alpha(t) \notin \Bp_z$ there exists 
  $b =\alpha(s) < \alpha(t)$ such that $\rho_z^+(\alpha(s)) \ge
  \rho_z^+(\alpha(t))$. Now, as $\alpha(s) < \alpha(t)$ it follows
  that $s < t$ and so, $t \notin \Bp_{z \circ \alpha}$, thereby proving (i).  

(ii).  First note that $\alpha^{-1}(\max \seq{S}) \in
  \alpha^{-1}(S)\cap \mathscr{B}_\alpha$, which proves
  $(\ge)$. Conversely, suppose $t \in \alpha^{-1}(S) \cap
  \mathscr{B}_\alpha$. Then
  \begin{align*}
    t &= \min\seq{s \in [0,\tau]\,:\, \alpha(s) \ge \alpha(t)} \\
    &  \le \min\seq{s \in [0,\tau]\,:\, \alpha(s) \ge \max \seq{S}} \\
    &= \alpha^{-1}(\max\seq{S}),
  \end{align*}
  which proves $(\le)$. 

(iii). Suppose $0 \le t \le \zeta^*$. By
  monotonicity, $t \in \alpha^{-1}([0,\tau^*])$, which proves
  $(\subset)$. Conversely, assume $t \in \alpha^{-1}([0,\tau^*]) \cap
  \mathscr{B}_\alpha$. Then
  \begin{equation*}
    t = \min\seq{s \,:\, \alpha(s) \ge \alpha(t)} \le \min\seq{s 
    \,:\, \alpha(s) \ge \tau^*} 
    = \alpha^{-1}(\tau^*) \le \zeta^*,
  \end{equation*}
  implying $(\supset)$.
\end{proof}

\begin{proof}[Proof of Lemma~\ref{lemma:OrmCompEquiv}]
  Let $\seq{\zeta_n}_{n \ge 1}^M$ be the interpolation points of
  $\orm_{\tau,M_\pm}(z \circ \alpha)$. By definition, $\zeta_0 =
  \tau$. Suppose $\tau \notin \Bp_{z \circ \alpha} \cup \Bm_{z \circ \alpha}$. 
  Applying Lemma~\ref{lemma:BpBmComp}, 
  \begin{align*}
    \zeta_1 &= \max\seq{(\Bp_{z \circ \alpha} \cup \Bm_{z \circ \alpha}) \cap [0,\tau]} \\
    &= \max\seq{\alpha^{-1}((\Bp_z \cup \Bm_z) \cap [0,\tau]) \cap \mathscr{B}_\alpha} \\
    &= \alpha^{-1}(\tau_1).
  \end{align*}
  Suppose $\zeta_n = \alpha^{-1}(\tau_n)$. If $\tau_n \in \Bpm_z$,
  then by Lemma~\ref{lemma:BpBmComp}~(i), $\zeta_n \in \Bpm_{z \circ \alpha}$. 
  By (i), (ii), and (iii),
  \begin{equation*}
   \zeta_{n+1} = \max \seq{\Bpm_{z \circ \alpha} \cap [0,\zeta_n]} 
	       = \max \seq{\alpha^{-1}(\Bpm_{z}\cap [0,\tau_n]) \cap \mathscr{B}_\alpha } 
	       = \alpha^{-1}(\tau_{n+1}).
  \end{equation*}
  Consequently $\zeta_n = \alpha^{-1}(\tau_n)$ for all $1 \le n \le N$. 
  As $\zeta_N = 0$ and $\alpha^{-1}(0) = 0$ it follows that $\zeta_N = 0$, and so $M = N$. 
  Consider the last statement. Let us denote 
  by $\mathscr{L}(\seq{(t_n,z_n)_{n = 0}^N})$ the piecewise 
  linear interpolation of the points $\seq{(t_n,z_n)_{n = 0}^N}$. 
  Set $\psi = \mathscr{L}(\seq{\tau_n,\zeta_n}_{n=0}^N)$. 
  By the above, $\alpha(\zeta_n) = \tau_n$ and so
  \begin{align*}
    \orm_{\tau,u_0}(z \circ \alpha)\circ\psi 
	  = \mathscr{L}\left(\seq{\zeta_n,z \circ \alpha(\zeta_n)}\right) \circ \psi 
	  &= \mathscr{L}\left(\seq{\tau_n,z \circ \alpha(\zeta_n)}\right) \\
	  &= \mathscr{L}\left(\seq{\tau_n,z(\tau_n)}\right) 
	  = \orm_{\tau,u_0}(z).
  \end{align*}
  As a consequence, cf.~Corollary~\ref{cor:PiecewiseEquiv}, it follows that
  $\orm_{\tau,u_0}(z \circ \alpha) \sim \orm_{\tau,u_0}(z)$. 
\end{proof}

\begin{proof}[Proof of Theorem \ref{thm:GeneralEquiv}]
  Let $\seq{\ue_0}_{\varepsilon > 0}$ be the approximation of $u_0$
  from Lemma~\ref{lemma:InitApprox}. Let us show that $z_1 \sim z_2$
  on $[0,\tau]$ with initial condition $\ue_0$. Note first that
  $\rho_{z_1,\varepsilon}^\pm \circ \alpha_1 =
  \rho_{z_2,\varepsilon}^\pm \circ \alpha_2$, where 
  $\rho_{z_1,\varepsilon}^\pm,\rho_{z_2,\varepsilon}^\pm$ are defined
  with respect to $M_-^\varepsilon,M_+^\varepsilon$, cf.~Lemma~\ref{lemma:InitApprox}. 
  By Theorem~\ref{thm:OrmEquiv} it
  suffices to prove $\orm_{\tau,\ue_0}(z_1) \sim
  \orm_{\tau,\ue_0}(z_2)$. However, by Lemma~\ref{lemma:OrmCompEquiv},
  we have
  \begin{equation*}
    \orm_{\tau,\ue_0}(z_1) \sim \orm_{\tau,\ue_0}(z_1 \circ \alpha_1)
    = \orm_{\tau,\ue_0}(z_2 \circ \alpha_2) \sim
    \orm_{\tau,\ue_0}(z_2). 
  \end{equation*}
  Let $\ue_i$ be the entropy solution to \eqref{eq:sscl} on
  $[0,\tau]$ with initial condition $\ue_0$ and path $z_i$, $i=1,2$. 
  By the above, $\ue_1(\tau) = \ue_2(\tau)$. Applying
  \eqref{eq:stabilityPes} and the triangle inequality it follows that
  \begin{align*}
    \norm{u_1(\tau)-u_2(\tau)}_{L^1(\R)}
    &\le \norm{u_1(\tau)-\ue_1(\tau)}_{L^1(\R)} + \norm{\ue_2(\tau)-u_2(\tau)}_{L^1(\R)} \\
    &\le 2\norm{\ue_0(\tau)-u_0(\tau)}_{L^1(\R)} \rightarrow 0,
  \end{align*}
  as $\varepsilon \downarrow 0$. Hence, $u_1(\tau) = u_2(\tau)$.
\end{proof}
Next, we turn to the proof of Theorem~\ref{thm:EqDist}. 
We first make some observations regarding the 
properties of the map $\Phi$ defined in Theorem~\ref{thm:EqDist}.

 \begin{lemma}\label{lemma:PhiProp}
  Let $\Phi$ be defined in Theorem~\ref{thm:EqDist} 
  and $z_1,z_2 \in C_0([0,\tau])$ be paths such 
  that $\orm_{\tau,M_\pm^i}(z_i)$, $i=1,2$, are well defined. 
  Then the following statements are true:
  \begin{itemize}
     \item[(i)] The map $\Phi$ is symmetric in the sense that 
     \begin{displaymath}
      \Phi[z_1,z_2](\alpha,\beta) = \Phi[z_2,z_1](\beta,\alpha), 
      \qquad \alpha,\beta \in \mathscr{A}_\tau.
     \end{displaymath}
     \item[(ii)] The map $\alpha \mapsto \Phi[z_1,z_2](\alpha,\beta)$, defined 
     for $\alpha,\beta \in \mathscr{A}_\tau$, is continuous with 
     respect to ``uniform convergence from above", i.e., 
     if $\seq{\alpha_j}_{j \geq 0} \subset \mathscr{A}_\tau$ 
     satisfies $\alpha_j \downarrow \alpha$ uniformly 
     on $[0,\tau]$ as $j \rightarrow \infty$, then
      \begin{displaymath}
	\Phi[z_1,z_2](\alpha_j,\beta) \rightarrow 
	\Phi[z_1,z_2](\alpha,\beta) \mbox{ as $j \rightarrow \infty$.}
      \end{displaymath}
     \item[(iii)] Suppose $\tilde{z}_1 \in C_0([0,\tau])$ satisfies 
     $\rho_{z_1}^\pm \circ \alpha_1 = \rho_{\tilde{z}_1}^\pm \circ \beta_1$ 
     for some $\alpha_1,\beta_1 \in \mathscr{A}_\tau$ and $\tilde{z}_1(\tau) = z_1(\tau)$. 
     Then
     \begin{displaymath}
     \Phi[z_1,z_2](\alpha_1 \circ \zeta,\beta) = \Phi[\tilde{z}_1,z_2](\beta_1 \circ \zeta,\beta), 
     \qquad \zeta,\beta \in \mathscr{A}_\tau.
     \end{displaymath}
  \end{itemize}

 \end{lemma}

 \begin{proof}
Set 
 \begin{displaymath}
 	\mathscr{S}^\pm[z_1,z_2](\alpha,\beta):=
	 \seq{\abs{\rho_{z_1}^\pm(\alpha(t))
	 -\rho_{z_2}^\pm(\beta(t))}\,:\,t \in \mathcal{T}^\pm_{z_1 \circ \alpha}
	 \cup \mathcal{T}^\pm_{z_2 \circ \beta}},
 \end{displaymath}
 so that 
 \begin{displaymath}
   \Phi[z_1,z_2](\alpha,\beta) = \max_\pm \seq{\mathscr{S}^\pm[z_1,z_2](\alpha,\beta) 
   \cup \seq{\abs{z_1(\tau)-z_2(\tau)}}}.
  \end{displaymath}

(i) The statement is a trivial consequence of the definition.

(ii)  By Lemma~\ref{lemma:OrmCompEquiv} and the 
fact that $\alpha_j(\alpha_j^{-1}(t)) = t$,
 \begin{align*}
 	\mathscr{S}^\pm[z_1,z_2](\alpha_j,\beta) &= 
	\seq{\abs{\rho_{z_1}^\pm(t)
	-\rho_{z_2}^\pm(\beta(\alpha_j^{-1}(t)))}\,:\,t \in \mathcal{T}^\pm_{z_1}} \\
	&\quad \cup \seq{\abs{\rho_{z_1}^\pm(\alpha_j(t))
	-\rho_{z_2}^\pm(\beta(t))}\,:\,t \in \mathcal{T}^\pm_{z_2 \circ \beta}}.
 \end{align*}
By assumption, as $j\to \infty$, $\alpha_j \downarrow \alpha$ 
uniformly on $[0,\tau]$. It remains to verify that 
$\alpha_j^{-1}(t) \uparrow \alpha^{-1}(t)$ for 
any $t \in [0,\tau]$. Let $\delta > 0$. By assumption 
there exists $j_0(\delta) \geq 0$ such that 
\begin{displaymath}
  \alpha(s) \leq \alpha_j(s) \leq \alpha(s) + \delta 
  \mbox{ for all $j \geq j_0(\delta)$ and $s \in [0,\tau]$.}
\end{displaymath}
Hence, for $j \geq j_0(\delta)$,
\begin{align*}
   \alpha^{-1}(t) &= \min\seq{0 \leq s \leq \tau\,:\,\alpha(s) \geq t} \\
   &\geq \min\seq{0 \leq s \leq \tau\,:\,\alpha_j(s) \geq t} = \alpha_j^{-1}(t) \\
&\geq \min\seq{0 \leq s \leq \tau\,:\,\alpha(s) \geq t-\delta} = \alpha^{-1}(t-\delta).
 \end{align*}
 The result now follows since $\alpha^{-1}$ is left-continuous.
 
(iii) Observe that
\begin{align*}
  \mathscr{S}^\pm[z_1,z_2](\alpha_1 \circ \zeta,\beta)&
  := \seq{\abs{\rho_{z_1}^\pm(\alpha_1 \circ \zeta(t))
  -\rho_{z_2}^\pm(\beta(t))}\,:\,t \in \mathcal{T}^\pm_{z_1 \circ \alpha_1 \circ \zeta} 
  \cup \mathcal{T}^\pm_{z_2 \circ \beta}} \\
  &\;= \seq{\abs{\rho_{\tilde{z}_1}^\pm(\beta_1 
  \circ \zeta(t))-\rho_{z_2}^\pm(\beta(t)}\,:\,t \in 
  \mathcal{T}^\pm_{z_1 \circ \alpha_1 
  \circ \zeta} \cup \mathcal{T}^\pm_{z_2 \circ \beta}}.
 \end{align*}
 By assumption, $\mathcal{T}^\pm_{z_1 \circ \alpha_1} 
 = \mathcal{T}^\pm_{\tilde{z}_1 \circ \beta_1}$ and 
 so, by Lemma~\ref{lemma:OrmCompEquiv}, 
 $\mathcal{T}^\pm_{z_1 \circ \alpha_1 \circ \zeta} 
 =\mathcal{T}^\pm_{\tilde{z}_1 \circ \beta_1 \circ \zeta}$. It follows that 
 \begin{displaymath}
\mathscr{S}^\pm[z_1,z_2](\alpha_1 \circ \zeta,\beta) 
= \mathscr{S}^\pm[\tilde{z}_1,z_2](\beta_1 \circ \zeta,\beta).
 \end{displaymath}
 The result now follows as $\tilde{z}_1(\tau) = z_1(\tau)$.
 \end{proof}
 
 \begin{lemma}\label{lemma:LinConstEquiv}
Let $z \in C_0([0,\tau])$ and suppose $\orm_{\tau,M_\pm}(z)$ is well defined. 
Let $\mathscr{S}^\pm$ be disjoint finite subsets of $(0,\tau)$ 
such that $\mathcal{T}_z^\pm \subset \mathscr{S}^\pm$. 
Let $\Gamma^\pm:\mathscr{S}^\pm \rightarrow \R^\pm$ satisfy 
  \begin{displaymath}
   \begin{cases}
	\Gamma^\pm(t) = \rho_z^\pm(t), \qquad t \in \mathcal{T}_z^\pm, \\
	\Gamma^+(t) \leq \rho_z^+(t), \qquad t \in \mathscr{S}^+,\\
	\Gamma^-(t) \geq \rho_z^-(t), \qquad t \in \mathscr{S}^-.
   \end{cases}
\end{displaymath}
Define
\begin{displaymath}
   \mathcal{I} = \seq{(t,\Gamma^+(t)), t \in \mathscr{S}^+} 
   \cup \seq{(t,\Gamma^-(t)), t \in \mathscr{S}^-} \cup \seq{(0,0), (\tau,z(\tau))}.
  \end{displaymath}
Let $\tilde{z}$ be the piecewise linear interpolation of 
$\mathcal{I}$, i.e., $\tilde{z} = \mathcal{L}(\mathcal{I})$. 
Then $\tilde{z} \sim z$ on $[0,\tau]$ for any initial value $u_0$ 
satisfying \eqref{eq:LowUpOnInitial}.
 \end{lemma}

\begin{proof}
Let $\seq{\tau_n}_{n \geq 0}^N = \mathcal{T}_z$ and 
$\seq{\tilde{\tau}_n}_{n \geq 0}^{\tilde{N}} = \mathcal{T}_{\tilde{z}}$. 
We want to show that $\mathcal{T}_{\tilde{z}} = \mathcal{T}_z$. The result then 
follows by Theorem~\ref{thm:OrmEquiv} as 
$\orm_{\tau,M_\pm}(z) = \orm_{\tau,M_\pm}(\tilde{z})$. 
First observe, for any $t \in \mathscr{S}^+$,
\begin{align*}
	\rho_{\tilde{z}}^+(t) 
	&= \max\seq{\frac{1}{M_-},\max_{0 \leq s \leq t}\seq{\tilde{z}(s)}}
	\\ & = \max\seq{\frac{1}{M_-},\max\seq{\Gamma^+(s)\,:\, s \in [0,t] 
	\cap \mathscr{S}^+}} \leq \rho_z^+(t),
\end{align*}
with equality whenever $t \in \mathcal{T}_z^+$. 
Recall that $\tau_1 := \max\seq{\Bp_{z} \cup \Bm_z \cap [0,\tau]}$, we 
want to show that $\tilde{\tau}_1 = \tau_1$. As $(\tau_1,\tau] \cap \Bp_z = \emptyset$ 
it follows by Lemma~\ref{lemma:BpBmProp}~(iii) that $\rho_{z}^+$ is 
constant on $[\tau_1,\tau]$. Consequently, by the above, the 
same holds for $\rho_{\tilde{z}}^+$. Similarly, $\rho_{\tilde{z}}^-$ is 
constant on $[\tau_1,\tau]$. Suppose $\tau_1 \in \Bp_z$ and 
let $t = \max\seq{\seq{0} \cup \mathscr{S}^+ \cup \mathscr{S}^- 
\cap [0,\tau_1)}$. As $\tau_1 \in \Bp_z$ we have $\rho_z^+(t) < \rho_z^+(\tau_1)$. 
By the above inequality,
\begin{displaymath}
   \partial_-\rho_{\tilde{z}}^+(\tau_1) 
   = \frac{\rho_{\tilde{z}}^+(\tau_1)-\rho_{\tilde{z}}^+(t)}{\tau_1-t} 
   \geq \frac{\rho_z^+(\tau_1)-\rho_z^+(t)}{\tau_1-t} > 0.
\end{displaymath}
Hence, by Lemma~\ref{lemma:BpBmProp}~(iv), $\tau_1 \in \Bp_{\tilde{z}}$. 
It does follow that $\tilde{\tau}_1 \geq \tau_1$ and so $\tilde{\tau}_1 = \tau_1$. 
Similar argument holds if $\tau_1 \in \Bm_z$. The proof now 
follows by induction, each step being similar to the one above.
\end{proof}

The next result contains the main ideas needed 
in the proof of Theorem~\ref{thm:EqDist}.
\begin{lemma}\label{lemma:PhiInf}
  Fix two paths $z_1,z_2 \in C_0([0,\tau])$ and let $\Phi$ be 
  the function defined in Theorem~\ref{thm:EqDist}. Then
  \begin{itemize}
   \item[(i)] $\Phi[z_1,z_2](\iota,\iota) \leq \norm{z_1-z_2}_\infty$.
   \item[(ii)] Suppose $\alpha_1,\alpha_2 \in \mathscr{A}_\tau$ 
   satisfy $\mathcal{T}_{z_1 \circ \alpha_1}^+ 
   \cap \mathcal{T}_{z_2 \circ \alpha_2}^- 
   =  \emptyset$, $\mathcal{T}_{z_1 \circ \alpha}^- 
   \cap \mathcal{T}_{z_2 \circ \alpha_2}^+ = \emptyset$, and 
   $\alpha_i^{-1}(\tau) < \tau$ for $i = 1,2$. Then there 
   exists $\tilde{z}_i$ with  $z_i \overset{\tau,M^i_\pm}{\sim} \tilde{z}_i$, 
   $i = 1,2$, and
    \begin{displaymath}
      \Phi[z_1,z_2](\alpha_1,\alpha_2) = \norm{\tilde{z}_1-\tilde{z}_2}_\infty.	
    \end{displaymath} 
  \end{itemize}
\end{lemma}

\begin{proof} For $a,b \in \R$, let $a \wedge b := \min\seq{a,b}$ 
and $a \vee b := \max\seq{a,b}$. Consider (i). Pick $t \in (\kappa^+,\tau]$ and 
suppose $\rho^+_{z_1}(t) > \rho^+_{z_2}(t)$. As $t > \kappa^+$,
\begin{displaymath}
  \rho_{z_1}^+(t) = \rho_{z_1}^+(t)\vee \rho_{z_2}^+(t) > 
  \frac{1}{M^1_-} \vee \frac{1}{M^2_-} \geq \rho_{z_1}^+(0).
\end{displaymath}
By Lemma~\ref{lemma:BpBmProp}~(iii), 
$(0,t] \cap \Bp_{z_1} \neq \emptyset$. 
Hence, $\xi := \max\seq{[0,t] \cap \Bp_{z_1}} > 0$. 
As $(\xi,t] \cap \Bp_{z_1} = \emptyset$, it follows by 
Lemma~\ref{lemma:BpBmProp}~(i) and (iii) 
that $z_1(\xi) = \rho^+_{z_1}(\xi) = \rho^+_{z_1}(t)$. 
Hence,
\begin{displaymath}
  \rho^+_{z_1}(t)-\rho^+_{z_2}(t) = z_1(\xi) - z_2(\xi) 
  + \underbrace{(z_2(\xi)-\rho^+_{z_2}(\xi))}_{\leq 0} 
  + \underbrace{(\rho^+_{z_2}(\xi)-\rho^+_{z_2}(t))}_{\leq 0}.
\end{displaymath}
A similar argument holds for $\rho^+_{z_1}(t) < \rho^+_{z_2}(t)$. 
Consequently,
 \begin{displaymath}
  \abs{\rho^+_{z_1}(t)-\rho^+_{z_2}(t)} 
  \leq \norm{z_1-z_2}_\infty, \qquad
  \mbox{ for any $t \in (\kappa^+,\tau]$.}
\end{displaymath}
Similarly, 
\begin{displaymath}
  \abs{\rho^-_{z_1}(t)-\rho^-_{z_2}(t)} \leq \norm{z_1-z_2}_\infty, 
  \mbox{ for any $t \in (\kappa^-,\tau]$},
\end{displaymath}
 which finishes the proof of (i).

 Consider (ii). Let $\zeta_i = z_i \circ \alpha_i$ for $i = 1,2$. We want to apply Lemma~\ref{lemma:LinConstEquiv} to construct suitable $\tilde{z}_1,\tilde{z}_2$. 
 Let $\mathscr{S}^\pm := \mathcal{T}_{\zeta_1}^\pm \cup \mathcal{T}_{\zeta_2}^\pm$. 
 By assumption and Lemma~\ref{lemma:BpBmProp}~(i),
 \begin{displaymath}
  \mathscr{S}^+ \cap \mathscr{S}^- 
  = (\mathcal{T}_{\zeta_1}^+ \cup \mathcal{T}_{\zeta_2}^+) 
  \cap (\mathcal{T}_{\zeta_1}^- \cup \mathcal{T}_{\zeta_2}^-) 
  = (\mathcal{T}_{\zeta_1}^+ \cap \mathcal{T}_{\zeta_2}^-) 
  \cup (\mathcal{T}_{\zeta_1}^- \cap \mathcal{T}_{\zeta_2}^+) = \emptyset.
 \end{displaymath}
 By Lemma~\ref{lemma:BpBmComp}~(i), $\Bpm_{z_i \circ \alpha_i} 
 = \alpha_i^{-1}(\Bpm_{z_i}) \cap \mathscr{B}_{\alpha_i}$. 
 As $\alpha_i^{-1}(\tau) < \tau$ so 
 that $\tau \notin \mathscr{B}_{\alpha_i}$ it follows that 
 $\tau \notin  \mathcal{T}_{\zeta_i}^\pm$ for $i = 1,2$. Let 
 \begin{displaymath}
  	\Gamma_j^+(t) = \begin{cases}
                   \rho_{\zeta_1}^+(t) \wedge \rho_{\zeta_2}^+(t),	 & t \in [0,\kappa^+], \\
		    \rho_{\zeta_j}^+(t),				 & t \in (\kappa^+,\tau], 
                  \end{cases}
  		\qquad
	  \Gamma_j^-(t) = \begin{cases}
                   \rho_{\zeta_1}^-(t) \vee \rho_{\zeta_2}^-(t),	 & t \in [0,\kappa^-], \\
		    \rho_{\zeta_j}^+(t),				 & t \in (\kappa^-,\tau]. 
                  \end{cases}
 \end{displaymath}
 In order to apply Lemma~\ref{lemma:LinConstEquiv}, it 
 remains to verify that $\Gamma^\pm_i(t) = \rho_{\zeta_i}^\pm(t)$ 
 for any $t \in \mathcal{T}_{\zeta_i}^\pm$, $i = 1,2$. We consider $\Gamma^+_1$, 
 the proof for the others are similar. It is obviously true, for any 
 $t \in \mathcal{T}_{\zeta_1}^+ \cap (\kappa^+,\tau]$. 
 Suppose $\frac{1}{M^1_-} \geq \frac{1}{M^2_-}$. Then $\rho_{\zeta_1}^+(\kappa^+) 
 = \frac{1}{M^1_-} = \rho_{\zeta_1}^+(0)$ and so $(0,\kappa^+]\cap \Bp_{\zeta_1} 
 = \emptyset$. Hence $\mathcal{T}_{\zeta_1}^+ \cap [0,\kappa^+] = \emptyset$, 
 and we are done. Suppose $\frac{1}{M^1_-} \leq \frac{1}{M^2_-}$. 
 Then, for any $t \in [0,\kappa^+]$,
 \begin{displaymath}
  \rho_{\zeta_2}^+(t) = \rho_{\zeta_2}(\kappa^+) 
  \geq \rho_{\zeta_1}^+(\kappa^+) \geq \rho_{\zeta_1}^+(t),
 \end{displaymath}
 and so $\rho_{\zeta_1}^+(t) \wedge \rho_{\zeta_2}^+(t) = \rho_{\zeta_1}^+(t)$, and 
 we are done. It remains  to verify that 
 \begin{displaymath}
  \norm{\tilde{z}_1-\tilde{z}_2}_\infty 
  	= \max\seq{\abs{\tilde{z}_1(t)-\tilde{z}_2(t)}\,:\,t 
	\in \mathcal{T}_{\zeta_1} \cup \mathcal{T}_{\zeta_2}}
  	= \Phi[z_1,z_2](\alpha_1,\alpha_2).
 \end{displaymath}
The first equality follows since $\tilde{z}_1,\tilde{z}_2$ are piecewise linear, while
the second equality follows as
\begin{align*}
  \max\seq{\abs{\tilde{z}_1(t)-\tilde{z}_2(t)}\,:\,t \in \mathcal{T}_{\zeta_1} 
  \cup \mathcal{T}_{\zeta_2}} 
    = &\max\seq{\abs{\Gamma_1^+(t)-\Gamma_2^+(t)}\,:\,t \in \mathscr{S}^+ 
    \cap (\kappa^+,\tau]} \\
    &\vee
    \max\seq{\abs{\Gamma_1^-(t)-\Gamma_2^-(t)}\,:\,t \in \mathscr{S}^- 
    \cap (\kappa^-,\tau])} \\
    &\vee
    \abs{z_1(\tau)-z_2(\tau)}.
\end{align*}
 This finishes the proof of (ii).
 \end{proof}

Before we turn to the proof of Theorem~\ref{thm:EqDist}, we need 

\begin{lemma}\label{lemma:UpperApprox}
Let $\alpha \in \mathscr{A}_\tau$ satisfy $\alpha(t) > 0$ for all $t \in (0,\tau]$. 
Then, for any $\xi \in \mathscr{A}_\tau$, there exists 
$\zeta^\varepsilon \in \mathscr{A}_\tau$ such that 
$\alpha \circ \zeta^\varepsilon \downarrow \xi$ as $\varepsilon \downarrow 0$ 
with respect to uniform convergence.
 \end{lemma} 

\begin{proof}
Let $\breve{\alpha}$ be the lower convex envelope of $\alpha$:
 \begin{displaymath}
  \breve{\alpha}(t) = \sup\seq{\beta(t)\,:\,\beta \leq \alpha \mbox{ and 
  $\beta$ is convex on $[0,\tau]$}}.
 \end{displaymath}
 Then $\breve{\alpha}$ is strictly increasing. Let $\xi \in \mathscr{A}$ and define
 \begin{displaymath}
  \alpha^\varepsilon = (1-\varepsilon)\alpha + \varepsilon \breve{\alpha} , \qquad 
  \xi = \alpha^\varepsilon \circ \zeta^\varepsilon, \qquad
  \xi^\varepsilon = \alpha \circ \zeta^\varepsilon.
 \end{displaymath}
 Then $\xi^\varepsilon-\xi = \varepsilon(\alpha-\breve{\alpha}) \circ \zeta^\varepsilon$ 
 so that $\xi \leq \xi^\varepsilon \leq \xi + 2\varepsilon \tau$.\\
 \end{proof}
 
 \begin{proof}[Proof of Theorem~\ref{thm:EqDist}]
  The theorem is a direct consequence of claims 1 and 2.

 \emph{Claim 1}: We have
  \begin{displaymath}
   \norm{[z_1]-[z_2]}_\infty
	= \inf\seq{\Phi[\tilde{z}_1,\tilde{z}_2](\iota,\iota)\,: 
	\tilde{z}_1 \sim z_1, \tilde{z}_2 \sim z_2}.
  \end{displaymath}
 By Lemma~\ref{lemma:PhiInf}~(i),
 \begin{align*}
  \norm{[z_1]_{\tau,M^1_\pm}-[z_2]_{\tau,M^2_\pm}}_\infty 
      &= \inf\seq{\norm{\tilde{z}_1-\tilde{z}_2}_\infty\,: 
      \tilde{z}_1 \overset{\tau,M^1_\pm}{\sim} z_1,
      \tilde{z}_2 \overset{\tau,M^2_\pm}{\sim} z_2}\\
      &\geq \underbrace{\inf\seq{\Phi[\tilde{z}_1,\tilde{z}_2](\iota,\iota)\,: 
      \tilde{z}_1 \overset{\tau,M^1_\pm}{\sim} z_1,
      \tilde{z}_2 \overset{\tau,M^2_\pm}{\sim} z_2}}_{I}.
 \end{align*}
 Let $\varepsilon > 0$. Pick $\tilde{z}_1,\tilde{z}_2 \in C_0([0,\tau])$ such 
 that $\tilde{z}_1 \overset{\tau,M^1_\pm}{\sim} z_1,\tilde{z}_2 
 \overset{\tau,M^2_\pm}{\sim} z_2$ satisfy
 \begin{displaymath}
  I + \varepsilon > \Phi[\tilde{z}_1,\tilde{z}_2](\iota,\iota).
 \end{displaymath}
 For any $\delta > 0$, define
 \begin{displaymath}
 	\alpha_i^\delta(t) 
	= 
	\begin{cases}
		(1+\delta)t,  	& t \in [0,\tau/(1+\delta)],\\
		\tau,			& t \in [\tau/(1+\delta),\tau],
	\end{cases} 
 \end{displaymath}
 for $i = 1,2$. Then, for any  $\delta_0,\eta_0 > 0$, 
 there exist $0< \delta < \delta_0, 0 < \eta < \eta_0$ such that  
 \begin{displaymath}
  \mathcal{T}_{\tilde{z}_1 \circ \alpha_1^\delta}^+ \cap \mathcal{T}_{\tilde{z}_2 
  \circ \alpha_2^\eta}^- = \emptyset, \qquad \mathcal{T}_{\tilde{z}_1 
  \circ \alpha_1^\delta}^- \cap \mathcal{T}_{\tilde{z}_2 \circ \alpha_2^\eta}^+ 
  = \emptyset,
 \end{displaymath}
 and $\alpha_i^{-1}(\tau) < \tau$ for $i = 1,2$.
Consequently, by Lemma~\ref{lemma:PhiProp}~(ii) 
and \ref{lemma:PhiInf}~(ii) there exist 
$\delta,\eta>0$ and $\tilde{z}_1,\tilde{z}_2$ with 
$\tilde{z}_1 \overset{\tau,M^1_\pm}{\sim} z_1,\tilde{z}_2 
\overset{\tau,M^2_\pm}{\sim} z_2$ such that  
\begin{displaymath}
  \abs{\Phi[\tilde{z}_1,\tilde{z}_2](\alpha_1^\delta,\alpha_2^\eta)
  -\Phi[\tilde{z}_1,\tilde{z}_2](\iota,\iota)} \leq \varepsilon, \qquad 
  \Phi[\tilde{z}_1,\tilde{z}_2](\alpha_1^\delta,\alpha_2^\eta) 
  = \norm{\tilde{z}_1-\tilde{z}_1}_\infty.
\end{displaymath}
As a result,
 \begin{displaymath}
  I + 2\varepsilon > \Phi[\tilde{z}_1,\tilde{z}_2](\alpha_1^\delta,\alpha_2^\eta) 
  = \norm{\tilde{z}_1^\alpha-\tilde{z}_1^\alpha}_\infty  
  \geq \norm{[z_1]_{\tau,M^1_\pm}-[z_2]_{\tau,M^2_\pm}}_\infty,
 \end{displaymath}
 finishing the proof of Claim~1. 

 \emph{Claim~2}: Define
 \begin{align*}
  \mathscr{I}_1[z_1,z_2] &:= 
  \seq{\Phi[z_1,z_2](\alpha,\beta)\,:\, \alpha,\beta \in \mathscr{A}},
  \\
  \mathscr{I}_2[z_1,z_2] &:= \seq{\Phi[\tilde{z}_1,\tilde{z}_2](\iota,\iota)\,:\, 
  \tilde{z}_1 \overset{\tau,M^1_\pm}{\sim_\circ} z_1, \tilde{z}_2 
  \overset{\tau,M^2_\pm}{\sim_\circ} z_2}.
 \end{align*}
 Then $\mathrm{cl}(\mathscr{I}_1[z_1,z_2]) = \mathscr{I}_2[z_1,z_2]$. 
 Therefore, $\inf \mathscr{I}_1 = \inf \mathscr{I}_2$.

We first prove that $\mathscr{I}_1[z_1,z_2] \subset \mathscr{I}_2[z_1,z_2]$. 
For any $s \in \mathscr{I}_1$ there exists 
$\alpha_1,\alpha_2 \in \mathscr{A}$ such that 
$s = \Phi[z_1,z_2](\alpha_1,\alpha_2)$. 
Let $\tilde{z}_i = z_i \circ \alpha_i$, $i = 1,2$. 
It follows that $\rho_{z_i}^\pm \circ \alpha_i 
= \rho_{\tilde{z}_i}^\pm \circ \iota$ for $i = 1,2$. 
Thus, $\tilde{z}_1 \overset{\tau,M^1_\pm}{\sim_\circ} z_1$ 
and $\tilde{z}_2 \overset{\tau,M^2_\pm}{\sim_\circ} z_2$. 
By Lemma~\ref{lemma:PhiProp} (i) and (ii), $s = \Phi[z_1,z_2](\alpha_1,\alpha_2) 
= \Phi[\tilde{z}_1,\tilde{z}_2](\iota,\iota) \in \mathscr{I}_2[z_1,z_2]$.
 
 Now, suppose $s \in \mathscr{I}_2$, that is, 
 $s = \Phi[\tilde{z}_1,\tilde{z}_2](\iota,\iota)$ for $\tilde{z}_1 
 \overset{\tau,M^1_\pm}{\sim_\circ} z_1$ and $\tilde{z}_2 
 \overset{\tau,M^2_\pm}{\sim_\circ} z_2$.  Hence, there 
 exist $\alpha_i,\beta_i \in \mathscr{A}_\tau$ such 
 that $\rho_{z_i}^\pm \circ \alpha_i = \rho_{\tilde{z}_i}^\pm \circ \beta_i$ for $i = 1,2$. 
 Since $0 \leq M_\pm^i < \infty$ we might as well assume that $\beta_i(0) < \beta_i(t)$ 
 for any $0 < t \leq \tau$. Hence, by Lemma~\ref{lemma:UpperApprox}, there exist 
 sequences $\seq{\zeta_i^\varepsilon}_{\varepsilon > 0} \subset \mathscr{A}_\tau$ 
 such that $\beta_i \circ \zeta_i^\varepsilon \downarrow \iota$ 
 uniformly as $\varepsilon \downarrow 0$ for $i = 1,2$. 
 By Lemma~\ref{lemma:PhiProp}~(i) and (iii),
 \begin{align*}
  	s^{\varepsilon,\eta} 
      &:= \Phi[\tilde{z}_1,\tilde{z}_2](\beta_1 
      \circ \zeta_1^{\varepsilon},\beta_2 \circ \zeta_2^{\eta}) \\
      &\;= \Phi[z_1,z_2](\alpha_1 \circ \zeta_1^{\varepsilon},\alpha_2 
      \circ \zeta_2^{\eta}) \in \mathscr{I}_1[z_1,z_2].
\end{align*}
 By Lemma~\ref{lemma:PhiProp}~(ii), there exist 
 subsequences $\seq{\varepsilon_j}_{j \geq 0},\seq{\eta_j}_{j \geq 0}$ 
 such that $\varepsilon_j,\eta_j \downarrow 0$ 
 and $s^{\varepsilon_j,\eta_j} \rightarrow s$, as $j \rightarrow \infty$.
 \end{proof}

\end{document}